\documentclass[a4paper,11pt]{article}
\usepackage{amsthm,amsmath,amssymb,amsfonts,bbm,color,geometry,epsfig,listings,hyperref,enumerate,comment,nicefrac,verbatim,multirow,titlesec}
\newtheorem{lemma}{Lemma}[section]

\newtheorem{theorem}[lemma]{Theorem}

\newtheorem{corollary}[lemma]{Corollary}

\newcommand{\1}{\ensuremath{\mathbbm{1}}}
\providecommand{\N}{{\ensuremath{\mathbbm{N}}}}
\providecommand{\Z}{{\ensuremath{\mathbbm{Z}}}}
\providecommand{\R}{{\ensuremath{\mathbbm{R}}}}

\providecommand{\E}{{\ensuremath{\mathbb{E}}}}
\renewcommand{\P}{{\ensuremath{\mathbb{P}}}}
\newcommand{\eps}{{\ensuremath{\varepsilon}}}

\newcommand{\funcF}{F}
\newcommand{\LipConst}{L}
\newcommand{\Var}{{\ensuremath{\operatorname{Var}}}}

\makeatletter
\newcommand{\vast}{\bBigg@{4}}
\newcommand{\Vast}{\bBigg@{5}}
\makeatother

\title{Multi-level Picard approximations of
high-dimensional semilinear parabolic\\
differential
equations with gradient-dependent nonlinearities}

\author{Martin Hutzenthaler$^{1}$ \& Thomas Kruse$^{1}$
\bigskip
\\
\small{$^1$ Faculty of Mathematics, University of Duisburg-Essen, Germany}
}

\begin{document}

\maketitle
\makeatletter
\let\@makefnmark\relax
\let\@thefnmark\relax
\@footnotetext{\emph{AMS 2010 subject classification:} 65M75}
\@footnotetext{\emph{Key words and phrases:}
curse of dimensionality, high-dimensional PDEs, high-dimensional nonlinear BSDEs, multi-level 
Picard iteration, multi-level Monte Carlo method, gradient-dependent nonlinearities
  }
\makeatother
\abstract{
Parabolic partial differential equations (PDEs) and backward stochastic differential equations (BSDEs)
have a wide range of applications.
In particular, high-dimensional PDEs with gradient-dependent nonlinearities
appear often in the state-of-the-art pricing and hedging of financial derivatives.
In this article we prove
that semilinear heat equations
with gradient-dependent nonlinearities
can be approximated under suitable assumptions 
with computational complexity
that grows polynomially both in the dimension and the reciprocal of the
accuracy.
}

\section{Introduction}
Parabolic partial differential equations (PDEs) and backward stochastic differential equations (BSDEs) are key ingredients in a number of models in physics and financial engineering; see, e.g., the references in~\cite{EHutzenthalerJentzenKruse2017}.
These applications often lead to stochastic optimization problems
which result in a semilinear or quasilinear PDE with a nonlinearity
depending on the gradient of the solution.
Moreover these PDEs are high-dimensional if the financial derivative depends
on a whole basket of underlyings.
So it is important to approximate the solutions of such PDEs
approximately at single space-time points
(the full solution function is presumeably hard to approximate in high dimensions;
cf.\ Theorem 1 in Heinrich~\cite{Heinrich2006} for the elliptic case).
The numerical analysis literature contains a multitude of approximation methods
for parabolic PDEs and BSDEs; see
the review
in~\cite{EHutzenthalerJentzenKruse2017}
and the recent article~\cite{EHanJentzen2017}.
However, to the best of our knowledge,
none of these methods except for the branching diffusion method
fulfills the requirement
that the computational complexity grows at most polynomially both in the
dimension and in the reciprocal of the accuracy; see Section 6
in~\cite{EHutzenthalerJentzenKruse2017}
for a detailed discussion.
The branching diffusion method
proposed in
\cite{Henry-Labordere2012,
Henry-LabordereTanTouzi2014,Henry-LabordereEtAl2016}
 meets this requirement.
However, not only is this method only applicable to a special class of PDEs,
it also requires the terminal/initial condition to be quite small
(see 
Subsection 6.7
in~\cite{EHutzenthalerJentzenKruse2017}
for a detailed discussion).

The recent article~\cite{EHutzenthalerJentzenKruse2016}
proposes
a family of approximation methods based on Picard approximations
and multi-level Monte Carlo methods; see also~\eqref{eq:def:U} below.
The simulation results
in~\cite{EHutzenthalerJentzenKruse2017}
suggest that these methods work satisfactory
for $100$-dimensional semilinear PDEs from applications.
In addition Corollary 3.18
in~\cite{EHutzenthalerJentzenKruse2016}
shows
under suitable regularity assumptions
on the exact solution for semilinear heat equations with gradient-independent
nonlinearities that the computational complexity is bounded by $O(d\eps^{-(4+\delta)})$
for any $\delta\in(0,\infty)$, where $d$ is the dimensionality
of the problem and $\eps\in(0,\infty)$ is the prescribed accuracy.
Generalizing the proof of Corollary 3.18
in~\cite{EHutzenthalerJentzenKruse2016}
to
the gradient-dependent case is nontrivial.
In particular, we were not able to derive an inequality
analogous to (56)
in~\cite{EHutzenthalerJentzenKruse2016}
involving a family of suitable
seminorms to which one could apply a discrete Gronwall inequality.

So it remained an open problem to prove mathematically
that semilinear PDEs with gradient-dependent nonlinearity
and general terminal/initial condition can be approximated
with a computational effort which grows at most polynomially
both in the dimension and in the reciprocal of the prescribed
accuracy.
In this article we solve this problem for the first time.
More precisely, Corollary~\ref{c:computational_effort_glob_error}
below shows
under suitable regularity assumptions
on the exact solution for semilinear heat equations with gradient-dependent
nonlinearities that the computational complexity of the multi-level Picard approximations~\eqref{eq:def:U}
is bounded by $O(d\eps^{-(4+\delta)})$
for any $\delta\in(0,\infty)$, where $d$ is the dimensionality
of the problem and $\eps\in(0,\infty)$ is the prescribed accuracy.

The structure of this article is as follows.
Subsection~\ref{ssec:notation} gathers notation that we frequently use.
In Section~\ref{sec:Picard} we introduce the setting which we 
consider throughout this article and, in particular,
the multilevel Picard approximations~\eqref{eq:def:U}
with Gau\ss-Legendre quadrature
rules given by~\eqref{eq:def_gauss_leg}.
The reason for choosing Gau\ss-Legendre quadrature rules
is the very fast convergence in case of sufficiently smooth
integrands; cf.\ Lemma~\ref{l:quad_error_glob} below.
Fast readers can then jump to
Corollary~\ref{c:computational_effort_glob_error},
which is the main result of this article.
For the proof of
Corollary~\ref{c:computational_effort_glob_error},
we first derive the (recursive) bound~\eqref{eq:global.estimate}
for the global error and then iterate this inequality
to obtain the (non-recursive) bound~\eqref{eq:aux_glob3}
for the global error.
Finally Lemma~\ref{l:ub_iterated_GL}
provides an upper bound for 
the iterated Gau\ss-Legendre integrals
over inverse square roots appearing in~\eqref{eq:aux_glob3}.

\subsection{Notation}\label{ssec:notation}
We denote by
$
  \langle \cdot, \cdot \rangle \colon
  \left(
    \cup_{ n \in \N }
    (
    \R^n \times \R^n
    )
  \right)
  \to
  [0,\infty)
$
the function that satisfies 
for all $ n \in \N $, $ v = ( v_1, \dots, v_n ) $, $ w = ( w_1, \dots, w_n ) \in \R^n $
that
$
  \langle
    v, w
  \rangle
    =
   \sum_{i=1}^n v_i w_i
$.
For every $p\in \N$ we denote by
$
  \left\| \cdot \right\|_p \colon
  \left(
    \cup_{ n \in \N }
    \R^n
  \right)
  \to 
  [0,\infty)
$
and 
$
  \left\| \cdot \right\|_\infty \colon
  \left(
    \cup_{ n \in \N }
    \R^n
  \right)
  \to 
  [0,\infty)
$
the functions that satisfy
for all $ n \in \N $, $ v = ( v_1, \dots, v_n ) $
that
$
  \left\| v \right\|_p
  =
  \big[ 
    \sum_{i=1}^n\left| v_i \right|^p
  \big]^{ 1 / p }
$
and 
$
  \left\| v \right\|_\infty
  =
  \max_{i=1,\ldots,n}|v_i|.
$
For every
topological space $(E,\mathcal E)$ we denote by $\mathcal{B}(E)$
the Borel-sigma-algebra on  $(E,\mathcal E)$.
For all measurable spaces $(A,\mathcal{A})$ and $(B,\mathcal{B})$
we denote by $\mathcal{M}(\mathcal A,\mathcal B)$ the set of $\mathcal{A}$/$\mathcal{B}$-measurable
functions from $A$ to $B$.
For every probability space $(\Omega,\mathcal{A},\P)$ we denote by
$\left\|\cdot\right\|_{L^2(\P;\R)}\colon\mathcal{M}(\mathcal{A},\mathcal{B}(\R))\to[0,\infty]$ the function that
 satisfies for all
$X\in\mathcal{M}(\mathcal{A},\mathcal{B}(\R))$ that
$\|X\|_{L^2(\P;\R)}=\sqrt{\E\!\left[|X|^2\right]}$.
We denote by $\tfrac{0}{0}$, $0\cdot\infty$, $0^0$, and $\sqrt{\infty}$
the extended real numbers given by
 $\tfrac{0}{0}=0$, $0\cdot\infty=0$, $0^0=1$, and $\sqrt{\infty}=\infty$. 
 For every $a\in(0,\infty)$ and every $b\in\R$
we denote by
$\tfrac{a}{0}$, $\tfrac{-a}{0}$, $0^{-a}$, $\tfrac{1}{0^a}$,
$\tfrac{b}{\infty}$, and $0^a$ the extended real numbers 
given by
$\tfrac{a}{0}=\infty$, $\tfrac{-a}{0}=-\infty$, $0^{-a}=\infty$, $\tfrac{1}{0^a}=\infty$, 
$\tfrac{b}{\infty}=0$, and $0^a=0$.
For every $A \subseteq \Z$, $a\colon A \to \R$,
    and $k\in \Z$ we denote by $\prod_{l=k}^{k-1}a(l)$ and $\sum_{l=k}^{k-1}a(l)$ the real numbers given by $\prod_{l=k}^{k-1}a(l)=1$ and $\sum_{l=k}^{k-1}a(l)=0$.

\section{Multi-level Picard approximations}\label{sec:setting.full.discretization}\label{sec:Picard}
Let 
$ T \in (0,\infty) $, 
$ d \in \N $, 
$ 
  g \in C^2(\R^d,\R)
$,
$
  \Theta = \cup_{ n \in \N } \R^n
$,
$L\in\R^{d+1}$,
$K\in \R^d$,
let
$
  ( 
    \Omega, \mathcal{F}, \P, ( \mathbb{F}_t )_{ t \in [0,T] } 
  )
$
be a stochastic basis,
let
$
  W^{ \theta } \colon [0,T] \times \Omega \to \R^d 
$,
$ \theta \in \Theta $,
be independent standard $(\mathbb{F}_t)_{t\in[0,T]}$-Brownian motions
with continuous sample paths,
let
$
  \funcF \colon 
  \mathcal{M}(\mathcal{B}([0,T]\times\R^d),\mathcal{B}(\R^{d+1}))
  \to
  \mathcal{M}(\mathcal{B}([0,T]\times\R^d),\mathcal{B}(\R))
$ satisfy for all $u_1,u_2\in
  \mathcal{M}(\mathcal{B}([0,T]\times\R^d),\mathcal{B}(\R^{d+1}))$,
  $r\in [0,T]$, $y\in \R^d$ that
  \begin{equation}  \begin{split}\label{eq:fLipschitz}
    \left| (\funcF(u_1)-\funcF(u_2))(r,y)\right|\leq 
    \sum_{\nu=1}^{d+1}
    \LipConst_\nu \left|\left(u_1(r, y)-u_2(r, y)\right)_{\nu}\right|,
  \end{split}     \end{equation}
  let $g\colon \R^d \to \R$ satisfy for all $x,y \in \R^d$ that
  \begin{equation}\label{eq:gLipschitz}
  |g(x)-g(y)|\leq \sum_{\alpha=1}^d K_\alpha |(x-y)_\alpha|,
  \end{equation}
let $u^{\infty}=(u^{\infty}(r,y))_{(r,y)\in[0,T]\times\R^d}\in C^{1,2}([0,T]\times\R^d,\R)$ satisfy
for all $r\in(0,T)$, $y\in \R^d$ that
$u^{\infty}(T,y)=g(y)$
and
\begin{equation}  \begin{split}\label{eq:PDE}
  (\tfrac{\partial}{\partial r}u^{\infty})(r,y)
+\tfrac{1}{2}(\Delta_y u^{\infty})(r,y)
 +(\funcF((u^{\infty},\nabla_y u^{\infty})))(r,y)=0,
\end{split}     \end{equation}
let ${\bf u^\infty} \in C([0,T]\times \R^d, \R^{d+1})$ satisfy for all 
$r\in[0,T]$, $y\in \R^d$ that
${\bf u^\infty}(r,y)=(u^\infty(r,y),\nabla_y u^\infty(r,y))$,
for every $n\in\N$ let
$(c_i^{n})_{i\in\{1,\ldots,n\}}\subseteq[-1,1]$
be the $n$ distinct roots of the Legendre polynomial $[-1,1]\ni x\mapsto \tfrac{1}{2^nn!}\tfrac{d^n}{dx^n}[(x^2-1)^n]\in\R$,
$q^{n,[a,b]}\colon[a,b]\to\R$ be the function which satisfies for all $t\in[a,b]$ that
\begin{equation}  \begin{split}\label{eq:def_gauss_leg}
  q^{n,[a,b]}(t)=\begin{cases}
 \int_a^b \left[\prod_{\substack{i\in\{1,\ldots,n\},\\ c_i^n\neq \frac{2t-(a+b)}{b-a}}}
  \tfrac{2x-(b-a)c_i^n-(a+b)}{2t-(b-a)c_i^n-(a+b)}\right]\,dx
  &\colon (a<b) \text{ and }\big(\frac{2t-(a+b)}{b-a}\in\{c_1^n,\ldots,c_n^n\}\big)\\
  0&\colon\text{else,}
  \end{cases}
\end{split}     \end{equation}
let
$ 
  ({\bf U}_{ n,M,Q}^{\theta })_{n,M,Q\in\Z,\theta\in\Theta}
  \subseteq\mathcal{M}(\mathcal{B}([0,T]\times\R^d)\otimes\mathcal{F},\mathcal{B}(\R\times\R^d))
$
satisfy
for all 
$
  n,M,Q \in \N
$,
$ \theta \in \Theta $,
$ (s,x) \in [0,T)\times \R^d $
that $
{\bf U}_{0,M,Q}^{\theta}(s,x)=0$ and
\begin{equation}  \begin{split}\label{eq:def:U}
  &{\bf U}_{n,M,Q}^{\theta}(s,x)
  =
  \big(
    g(x)
    , 0
  \big)
  +
  \frac{1}{M^n}\sum_{i=1}^{M^n}(g(x+W^{(\theta,0,-i)}_T-W^{(\theta,0,-i)}_s)-g(x))
  \Big(
  1 ,
  \tfrac{ 
  W^{(\theta, 0, -i)}_{T}- W^{(\theta, 0, -i)}_{s}
  }{ T - s }
  \Big)
  \\
  &+\sum_{l=0}^{n-1}\sum_{t\in(s,T)}\frac{q^{Q,[s,T]}(t)}{M^{n-l}}\sum_{i=1}^{M^{n-l}}
  \big(\funcF({\bf U}_{l,M,Q}^{(\theta,l,i,t)})-\1_{\N}(l)\funcF( {\bf U}_{l-1,M,Q}^{(\theta,-l,i,t)})\big)
  (t,x+W_{t}^{(\theta,l,i)}-W_s^{(\theta,l,i)})
  \Big(
  1 ,
  \tfrac{ 
  W^{(\theta, l, i)}_{t}- W^{(\theta, l, i)}_{s}
  }{ t - s }
  \Big).
\end{split}     \end{equation}

\section{Preliminary results for Gau\ss-Legendre quadrature rules}

\begin{lemma}[Iterated Gau\ss-Legendre integration]\label{l:iterated_GL}
Assume the setting in Section~\ref{sec:setting.full.discretization} and let $Q\in \N$. Then it holds for all $k\in \N$, $t_0\in [0,T)$ that
\begin{equation}\label{eq:iterated_GL}
\sum_{\substack{t_1,\ldots,t_{k-1},t_{k}\in\R,\\t_0<t_1<\ldots<t_{k-1}<t_{k}< T}}
   \left[\prod_{i=0}^{k-1} \frac {q^{Q,[t_{i},T]}(t_{i+1})}{\sqrt{t_{i+1}-t_{i}}}\right]=(T-t_0)^{\nicefrac k2}\prod_{i=0}^{k-1} \left[\sum_{s\in (0,1)}q^{Q,[0,1]}(s)\frac{(1-s)^{\nicefrac i2}}{\sqrt{s}}\right].
\end{equation}
\end{lemma}
\begin{proof}[Proof of Lemma \ref{l:iterated_GL}]
First observe that for all $t_0\in [0,T)$ and $s\in [0,1]$ with $2s-1\in \{c_1^Q, c_2^Q, \ldots c^Q_Q\}$ the definition \eqref{eq:def_gauss_leg} and the integral transformation theorem with the substitution 
$[t_0,T]\ni x\mapsto \tfrac{x-t_0}{T-t_0}\in[0,1]$ prove that
\begin{equation}
\begin{split}
q^{Q,[t_0,T]}(s(T-t_0)+t_0)&=\int_{t_0}^T \left[\prod_{\substack{i\in\{1,\ldots,n\},\\ c_i^Q\neq \frac{2s(T-t_0)+2t_0-(t_0+T)}{T-t_0}}}
  \tfrac{2x-(T-t_0)c_i^Q-(t_0+T)}{2s(T-t_0)+2t_0-(T-t_0)c_i^Q-(t_0+T)}\right]\,dx\\
  &=\int_{t_0}^T  \left[\prod_{\substack{i\in\{1,\ldots,n\},\\ c_i^Q\neq 2s-1}}
  \tfrac{2(x-t_0)-(T-t_0)c_i^Q-(T-t_0)}{(T-t_0)(2s-c_i^Q-1)}\right]\,dx\\
  &=(T-t_0)\int_{0}^1\left[\prod_{\substack{i\in\{1,\ldots,n\},\\ c_i^Q\neq 2s-1}}
  \tfrac{2y-c_i^Q-1}{2s-c_i^Q-1}\right]\,dy\\
  &=(T-t_0)q^{Q,[0,1]}(s).
\end{split}
\end{equation}
This and \eqref{eq:def_gauss_leg} show that for all $t_0\in [0,T)$ and $s\in [0,1]$ it holds that
\begin{equation}\label{eq:scaling_gl_weights}
q^{Q,[t_0,T]}(s(T-t_0)+t_0)=(T-t_0)q^{Q,[0,1]}(s).
\end{equation}
We prove \eqref{eq:iterated_GL} by induction on $k\in \N$. For the base case $k=1$ observe that \eqref{eq:scaling_gl_weights} ensures that for all $t_0\in [0,T)$ it holds that
\begin{equation}
\sum_{t_1\in (t_0,T)} \frac {q^{Q,[t_{0},T]}(t_{1})}{\sqrt{t_{1}-t_{0}}}=\sum_{s\in (0,1)} \frac{q^{Q,[t_{0},T]}(s(T-t_0)+t_0)}{\sqrt{s(T-t_0)}}
=(T-t_0)^{\nicefrac 12}\sum_{s\in (0,1)}\frac{q^{Q,[0,1]}(s)}{\sqrt{s}}.
\end{equation} 
This establishes \eqref{eq:iterated_GL} in the base case $k=1$.
For the induction step $\N \ni k \rightarrow k+1 \in \N$ observe that the induction hypothesis implies that for all $t_0\in [0,T)$ it holds that
\begin{equation}
\begin{split}
&\sum_{\substack{t_1,\ldots,t_k,t_{k+1}\in\R,\\t_0<t_1<\ldots<t_k<t_{k+1}< T}}
   \left[\prod_{i=0}^{k} \frac {q^{Q,[t_{i},T]}(t_{i+1})}{\sqrt{t_{i+1}-t_{i}}}\right]
   =\sum_{t_1\in (t_0,T)}  \frac{q^{Q,[t_{0},T]}(t_1)}{\sqrt{t_1-t_0}}\left\{\sum_{\substack{t_2,\ldots,t_k,t_{k+1}\in\R,\\t_1<t_2<\ldots<t_k<t_{k+1}< T}}
   \left[\prod_{i=1}^{k} \frac {q^{Q,[t_{i},T]}(t_{i+1})}{\sqrt{t_{i+1}-t_{i}}}\right]\right\}\\
   &=\sum_{t_1\in (t_0,T)}  \frac{q^{Q,[t_{0},T]}(t_1)}{\sqrt{t_1-t_0}}\left\{(T-t_1)^{\nicefrac k2}\prod_{i=0}^{k-1} \left[\sum_{s\in (0,1)}q^{Q,[0,1]}(s)\frac{(1-s)^{\nicefrac i2}}{\sqrt{s}}\right]\right\}\\
   &=\left\{\prod_{i=0}^{k-1} \left[\sum_{s\in (0,1)}q^{Q,[0,1]}(s)\frac{(1-s)^{\nicefrac i2}}{\sqrt{s}}\right]\right\}
   \left\{ \sum_{t_1\in (t_0,T)}  q^{Q,[t_{0},T]}(t_1) \frac{(T-t_1)^{\nicefrac k2}}{\sqrt{t_1-t_0}}\right\}.
   \end{split}
\end{equation}
This together with \eqref{eq:scaling_gl_weights} ensures that for all $t_0\in [0,T)$ it holds that
\begin{equation}
\begin{split}
&\sum_{\substack{t_1,\ldots,t_k,t_{k+1}\in\R,\\t_0<t_1<\ldots<t_k<t_{k+1}< T}}
   \left[\prod_{i=0}^{k} \frac {q^{Q,[t_{i},T]}(t_{i+1})}{\sqrt{t_{i+1}-t_{i}}}\right]\\
   &=\left\{\prod_{i=0}^{k-1} \left[\sum_{s\in (0,1)}q^{Q,[0,1]}(s)\frac{(1-s)^{\nicefrac i2}}{\sqrt{s}}\right]\right\}
   \left\{ \sum_{s\in (0,1)}  q^{Q,[t_{0},T]}(s(T-t_0)+t_0)\, \frac{(T-s(T-t_0)-t_0)^{\nicefrac k2}}{\sqrt{s(T-t_0)}}\right\}\\
   &=(T-t_0)^{\nicefrac{(k+1)}{2}}\prod_{i=0}^{k} \left[\sum_{s\in (0,1)}q^{Q,[0,1]}(s)\frac{(1-s)^{\nicefrac i2}}{\sqrt{s}}\right]
   .
   \end{split}
\end{equation}
 This finishes the induction step $\N_0 \ni k \rightarrow k+1 \in \N$. Induction hence establishes \eqref{eq:iterated_GL}. The proof of Lemma~\ref{l:iterated_GL} is thus completed.
\end{proof}

\begin{lemma}\label{l:ub_frac_int}
Assume the setting in Section~\ref{sec:setting.full.discretization} and let $Q\in \N$, $j\in \N_0$. Then it holds that
\begin{equation}\label{eq:ub_frac_int}
\sum_{s\in (0,1)}q^{Q,[0,1]}(s)\frac{(1-s)^{j}}{\sqrt{s}}\le \frac{\Gamma(\frac 12) \Gamma(j+1)}{\Gamma(j+\frac 32)}.
\end{equation}
\end{lemma}
\begin{proof}[Proof of Lemma \ref{l:ub_frac_int}]
The Leibniz formula ensures that for all $\eps \in (0,\infty)$, $s\in (0,1)$ it holds that
\begin{equation}\label{eq:leibniz}
\begin{split}
\frac{d^{2Q}}{ds^{2Q}}\frac{(1-s)^{j}}{\sqrt{s+\eps}}&=\sum_{k=0}^{2Q}\binom{2Q}{k} \left[\frac{d^{2Q-k}}{ds^{2Q-k}}\frac{1}{\sqrt{s+\eps}}\right]\left[\frac{d^{k}}{ds^{k}}(1-s)^j\right]\\
&=\sum_{k=0}^{2Q}\binom{2Q}{k} \left[(s+\eps)^{-(2Q-k+\nicefrac 12)}\prod_{l=0}^{2Q-k-1}(-\tfrac 12 -l)\right]\left[(-1)^k(1-s)^{j-k}\prod_{l=0}^{k-1}(j-l)\right]\\
&=\sum_{k=0}^{\min\{j,2Q\}}\binom{2Q}{k} \left[(s+\eps)^{-(2Q-k+\nicefrac 12)}\prod_{l=0}^{2Q-k-1}(\tfrac 12 +l)\right]\left[(1-s)^{j-k}\prod_{l=0}^{k-1}(j-l)\right]\\
&\ge 0.
\end{split}
\end{equation}
The error representation for the Gau\ss-Legendre quadrature rule
 (see, e.g., \cite[Display (2.7.12)]{davis2007methods}) implies that for every $\eps \in (0,\infty)$ there exists $\xi \in (0,1)$ such that it holds that
 \begin{equation}
 \begin{split}
 \sum_{s\in (0,1)}q^{Q,[0,1]}(s)\frac{(1-s)^{j}}{\sqrt{s+\eps}}&=\int_0^1\frac{(1-s)^{j}}{\sqrt{s+\eps}}\,ds-\frac{(Q!)^4}{(2Q+1)[(2Q)!]^3}\frac{d^{2Q}}{ds^{2Q}}\Bigg|_{s=\xi}\frac{(1-s)^{j}}{\sqrt{s+\eps}}.
 \end{split}
 \end{equation}
 This and \eqref{eq:leibniz} prove that for all $\eps \in (0,\infty)$ it holds that
 \begin{equation}\label{eq:ub_frac_int_fin}
 \begin{split}
 \sum_{s\in (0,1)}q^{Q,[0,1]}(s)\frac{(1-s)^{j}}{\sqrt{s+\eps}}\le \int_0^1\frac{(1-s)^{j}}{\sqrt{s+\eps}}\,ds
 \le \int_0^1\frac{(1-s)^{j}}{\sqrt{s}}\,ds = \frac{\Gamma(\frac 12) \Gamma(j+1)}{\Gamma(j+\frac 32)}.
 \end{split}
 \end{equation}
 Letting $\eps \to 0$ in \eqref{eq:ub_frac_int_fin}
  completes the proof of Lemma \ref{l:ub_frac_int}.
\end{proof}

\begin{lemma}[Upper bound for iterated Gau\ss-Legendre integration]\label{l:ub_iterated_GL}
Assume the setting in Section~\ref{sec:setting.full.discretization} and let $Q\in \N$. Then it holds for all $k\in \N$, $t_0\in [0,T)$ that
\begin{equation}\label{eq:ub_iterated_GL}
\sum_{\substack{t_1,\ldots,t_{k-1},t_{k}\in\R,\\t_0<t_1<\ldots<t_{k-1}<t_{k}< T}}
   \left[\prod_{i=0}^{k-1} \frac {q^{Q,[t_{i},T]}(t_{i+1})}{\sqrt{t_{i+1}-t_{i}}}\right]\le \frac{2((T-t_0)\pi)^{\nicefrac k2}}{\Gamma(\frac{k}{2})}.
\end{equation}
\end{lemma}
\begin{proof}[Proof of Lemma \ref{l:ub_iterated_GL}]
Throughout this proof let $w\colon \N \to \R$ be the function that satisfies for all $k\in \N$ that $w(k)=\prod_{i=0}^{k-1} \frac{\Gamma(\lfloor \frac i2 \rfloor +1)}{\Gamma(\lfloor \frac i2 \rfloor +\frac 32)}$.
First observe that for all $k\in \{2n\colon n\in \N\}$ it holds that
\begin{equation}\label{eq:aux_inequal1}
\frac{\Gamma(\frac{k+1}{2})\Gamma(\lfloor \frac k2 \rfloor +1)}{\Gamma(\frac{k}{2})\Gamma(\lfloor \frac k2 \rfloor +\frac 32)}=\frac{\Gamma(\frac{k+1}{2})\Gamma(\frac k2  +1)}{\Gamma(\frac{k}{2})\Gamma(\frac k2 +\frac 32)}
=\frac{\Gamma(\frac{k}{2}+\frac 12 )\Gamma(\frac k2)\frac{k}{2}}{\Gamma(\frac{k}{2})\Gamma(\frac k2 + \frac 12)(\frac k2 + \frac 12)}
=\frac{\frac k2 }{\frac{k}{2}+\frac 12}\le 1.
\end{equation}
Moreover, the fact that $\Gamma\colon (0,\infty)\to (0,\infty)$ is logarithmically convex ensures that for all $k\in \{2n-1\colon n\in \N\}$ it holds that
\begin{equation}
\frac{\Gamma(\frac{k+1}{2})\Gamma(\lfloor \frac k2 \rfloor +1)}{\Gamma(\frac{k}{2})\Gamma(\lfloor \frac k2 \rfloor +\frac 32)}=\frac{\Gamma(\frac{k}{2}+\frac 12)^2}{\Gamma(\frac{k}{2})\Gamma(\frac {k}2 +1)}
\le 1.
\end{equation}
This and \eqref{eq:aux_inequal1} prove that for all $k\in \N$ it holds that
\begin{equation}\label{eq:aux_inequal2}
\frac{\Gamma(\frac{k+1}{2})\Gamma(\lfloor \frac k2 \rfloor +1)}{\Gamma(\frac{k}{2})\Gamma(\lfloor \frac k2 \rfloor +\frac 32)}
\le 1.
\end{equation}
Next we show that for all $k\in \N$ it holds that
\begin{equation}\label{eq:ub_prod}
w(k)\le \frac{2}{\Gamma(\frac{k}{2})}.
\end{equation}
We prove \eqref{eq:ub_prod} by induction on $k\in \N$. For the base case $k=1$ we note that it holds that
\begin{equation}
w(1)=\frac{\Gamma(1)}{\Gamma(\frac 32)}=\frac{2}{\Gamma(\frac 12)}.
\end{equation}
This establishes \eqref{eq:ub_prod} in the base case $k=1$. 
For the induction step $\N \ni k \to k+1\in \N$ observe that the induction hypothesis and \eqref{eq:aux_inequal2} show that
\begin{equation}
w(k+1)=w(k)\frac{\Gamma(\lfloor \frac k2 \rfloor +1)}{\Gamma(\lfloor \frac k2 \rfloor +\frac 32)}\le \frac{2\Gamma(\lfloor \frac k2 \rfloor +1)}{\Gamma(\frac{k}{2})\Gamma(\lfloor \frac k2 \rfloor +\frac 32)}=\frac{\Gamma(\frac{k+1}{2})\Gamma(\lfloor \frac k2 \rfloor +1)}{\Gamma(\frac{k}{2})\Gamma(\lfloor \frac k2 \rfloor +\frac 32)}\frac{2}{\Gamma(\frac{k+1}{2})}\le \frac{2}{\Gamma(\frac{k+1}{2})}.
\end{equation}
This finishes the induction step $\N \ni k \rightarrow k+1 \in \N$. Induction hence establishes \eqref{eq:ub_prod}. Lemma \ref{l:iterated_GL}, Lemma \ref{l:ub_frac_int}, the facts that $\forall s\in (0,1)\colon q^{Q,[0,1]}(s)\ge 0$ (see, e.g., \cite[Section 2.7]{davis2007methods}) and $\Gamma(\tfrac 12)=\sqrt{\pi}$, and \eqref{eq:ub_prod} show that for all $k\in \N$, $t_0\in [0,T)$ it holds that
\begin{equation}
\begin{split}
\sum_{\substack{t_1,\ldots,t_{k-1},t_{k}\in\R,\\t_0<t_1<\ldots<t_{k-1}<t_{k}< T}}
   \left[\prod_{i=0}^{k-1} \frac {q^{Q,[t_{i},T]}(t_{i+1})}{\sqrt{t_{i+1}-t_{i}}}\right] &=(T-t_0)^{\nicefrac k2}\prod_{i=0}^{k-1} \left[\sum_{s\in (0,1)}q^{Q,[0,1]}(s)\frac{(1-s)^{\nicefrac i2}}{\sqrt{s}}\right]\\
   & \le (T-t_0)^{\nicefrac k2}\prod_{i=0}^{k-1} \left[\sum_{s\in (0,1)}q^{Q,[0,1]}(s)\frac{(1-s)^{\lfloor \nicefrac i2 \rfloor}}{\sqrt{s}}\right]\\
   & \le (T-t_0)^{\nicefrac k2}\prod_{i=0}^{k-1} \left[\frac{\Gamma(\frac 12) \Gamma(\lfloor \frac{i}{2} \rfloor +1)}{\Gamma(\lfloor \frac{i}{2} \rfloor+\frac 32)}\right]\\
   & = (T-t_0)^{\nicefrac k2}\Gamma(\tfrac 12)^k w(k)\\
   &\le \frac{2((T-t_0)\pi)^{\nicefrac k2}}{\Gamma(\frac{k}{2})}.
\end{split}
\end{equation}   
   This completes the proof of Lemma \ref{l:ub_iterated_GL}.
\end{proof}

\begin{lemma}[Iterated sums]\label{l:iterated_sum}
Let $n\in \N$, $l_0\in \{0,\ldots,n-1\}$, and $j\in \{1,\ldots,n-l_0-1\}$. Then it holds that
\begin{equation}\label{eq:ub_iterated_sum}
\sum_{\substack{l_1,\ldots,l_{j}\in\N,\\l_0<l_1<\ldots<l_{j}< n}} 1 =\binom{n-l_0-1}{j}.
\end{equation}
\end{lemma}
\begin{proof}[Proof of Lemma \ref{l:iterated_sum}]
The natural number $\sum_{\substack{l_1,\ldots,l_{j}\in\N,\\l_0<l_1<\ldots<l_{j}< n}} 1$ is the number of ways to choose a subset of size $j$ elements from a set of $n-l_0-1$ elements. This completes the proof of Lemma \ref{l:iterated_sum}.
\end{proof}

\begin{lemma}[Log-subadditivity]\label{l:ub_poly}
Let $d,p\in \N$, $x,y \in \R^d$, and let $\|\cdot \|\colon \R^d\to [0,\infty)$ be a norm. Then $1+\|x+y\|^p \le (1+\|y\|)^p(1+\|x\|^p)$.
\end{lemma}  
  \begin{proof}[Proof of Lemma \ref{l:ub_poly}]
  It holds that
  \begin{equation}
  \begin{split}
  1+\|x+y\|^p&\le 1+(\|x\|+\|y\|)^p
 =1+\sum_{k=0}^p\binom{p}{k}\|x\|^{p-k}\|y\|^k
  =(1+\|x\|^p)\left(1+\sum_{k=1}^p\binom{p}{k}\tfrac{\|x\|^{p-k}}{1+\|x\|^p}\|y\|^k \right)\\
  &\leq (1+\|x\|^p)\left(1+\sum_{k=1}^p\binom{p}{k}\|y\|^k \right)
  =(1+\|y\|)^p(1+\|x\|^p).
  \end{split}
  \end{equation}
  This completes the proof of Lemma \ref{l:ub_poly}.
  \end{proof}

\section{Error analysis for multi-level Picard approximations with Gau\ss-Legendre quadrature rules}
\begin{lemma}[Approximations are integrable]\label{l:approximations.integrable}
 Assume the setting in Section~\ref{sec:setting.full.discretization},
let $p,M,Q\in\N$ and assume for all $t\in[0,T]$ that 
\begin{equation}\label{eq:approximations.integrable.assump}
\sup_{x\in \R^d} 
\frac{ 
  \left|
    g(x)
    \right|
    }
    {1+\|x\|^p_1}+\sup_{x\in \R^d}
    \frac{\left| 
  \big(\funcF(0)\big)
  (t,x)\right|}
  {1+\|x\|^p_1}<\infty.
\end{equation}
Then 
\begin{enumerate} [(i)] 
\item \label{item:approximations.integrable.i} for all
$n\in\N_0$, $\theta\in\Theta$, $s\in[0,T)$, $\nu \in \{1,\ldots,d+1\}$ it holds that
\begin{equation}\label{eq:approximations.integrable}
  \E\!\left[\sup_{x\in \R^d}\frac{\left|\left({\bf U}_{n,M,Q}^{\theta}(s,x)\right)_\nu\right|}{1+\|x\|^p_1}\right]<\infty,
  \end{equation}
\item \label{item:approximations.integrableiia} for all $n\in \N$, $\theta \in \Theta$, $s\in [0,T)$, $t\in (s,T]$, $x\in \R^d$, $\nu \in \{1,\ldots,d+1\}$ it holds that
\begin{equation}
 \E\left[\left|
  \big(\funcF({\bf U}_{n,M,Q}^{\theta})\big)
  (t,x+W_{t}^{0}-W_s^{0})
  \Big(
  1 ,
  \tfrac{ 
  W^{0}_{t}- W^{0}_{s}
  }{ t - s }
  \Big)_\nu\right|\right]<\infty,
  \end{equation}
and
\item \label{item:approximations.integrable.ii} 
for all  $n\in\N$, $\theta \in \Theta$, $s\in[0,T)$, $x\in\R^d$
   it holds that
  \begin{equation}  \begin{split}\label{eq:discreteFeynmanKac}
    \E\!\left[{\bf U}_{n,M,Q}^{\theta}(s,x)\right]&=\E\!\left[
    g(x+W_T^0-W_s^0)
     \Big(
     1 ,
     \tfrac{ 
     W^{0}_{T}- W^{0}_{s}
     }{ T - s }
     \Big)
    \right]
    \\&+
    \E\!\left[
    \sum_{t\in(s,T)}q^{Q,[s,T]}(t)
    \left(\funcF( {\bf U}_{n-1,M,Q}^{\theta})\right)\!(t,x+W_t^0-W_s^0)
     \Big(
     1 ,
     \tfrac{ 
     W^{0}_{t}- W^{0}_{s}
     }{ {t - s} }
     \Big)
    \right].
  \end{split}     \end{equation}
\end{enumerate}
\end{lemma}
\begin{proof}[Proof of Lemma~\ref{l:approximations.integrable}]
  We prove \eqref{item:approximations.integrable.i} by induction on $n\in \N_0$. The induction base $n=0$ is clear.  For the induction step $\N_0 \ni n \rightarrow n+1 \in \N$,
 let $n\in \N_0$ and assume that \eqref{item:approximations.integrable.i} holds for $n=0$, $n=1$, $\ldots$, $n=n$. 
 The triangle inequality, Lemma \ref{l:ub_poly}, \eqref{eq:gLipschitz}, and \eqref{eq:fLipschitz} ensure that for all $\theta\in \Theta$, $s\in [0,T)$, $\nu \in \{1,\ldots,d+1\}$ it holds that
 \begin{equation}\label{eq:approx_int1}
 \begin{split}
  &\E\!\left[\sup_{x\in \R^d}\tfrac{\left|\left({\bf U}_{n+1,M,Q}^{\theta}(s,x)\right)_\nu\right|}{1+\|x\|^p_1}\right]
  \leq
  \sup_{x\in \R^d} 
  \tfrac{\left|\big(
    g(x)
    , 0
  \big)_\nu\right|}
  {1+\|x\|^p_1}
  +
  \E\!\left[\sup_{x\in \R^d}
  \tfrac{\left|
  (g(x+W^{0}_T-W^{0}_s)-g(x))
  \right|}
  {1+\|x\|^p_1}
  \left|
  \Big(
  1 ,
  \tfrac{ 
  W^{(\theta, 0, -i)}_{T}- W^{(\theta, 0, -i)}_{s}
  }{ T - s }
  \Big)_\nu
  \right|\right]
  \\
  &+\sum_{l=0}^{n}\sum_{t\in(s,T)}q^{Q,[s,T]}(t)
\E\!\left[\sup_{x\in \R^d}
\tfrac{\left| 
  \big(\funcF({\bf U}_{l,M,Q}^{(\theta,l,1,t)})-\1_{\N}(l)\funcF( {\bf U}_{l-1,M,Q}^{(\theta,-l,1,t)})\big)
  (t,x+W_{t}^{(\theta,l,1)}-W_s^{(\theta,l,1)})\right|}
  {1+\|x\|^p_1}
  \left|
  \Big(
  1 ,
  \tfrac{ 
  W^{(\theta, l, 1)}_{t}- W^{(\theta, l, 1)}_{s}
  }{ t - s }
  \Big)_\nu
  \right|\right]\\
  &\leq
  \sup_{x\in \R^d} 
  \tfrac{\left|
    g(x)
    \right|}
    {1+\|x\|^p_1}
  +\sum_{\alpha=1}^dK_\alpha 
  \E\!\left[\left|
  (W^{0}_T-W^{0}_s)_\alpha
  \Big(
  1 ,
  \tfrac{ 
  W^{(\theta, 0, -i)}_{T}- W^{(\theta, 0, -i)}_{s}
  }{ T - s }
  \Big)_\nu
  \right|\right]
  \\
  &+\sum_{l=1}^{n}\sum_{t\in(s,T)}q^{Q,[s,T]}(t)
  \sum_{\nu_1=1}^{d+1}
  L_{\nu_1}
\E\Bigg[\sup_{x\in \R^d}
\tfrac{\left| 
  \big({\bf U}_{l,M,Q}^{(\theta,l,1,t)}- {\bf U}_{l-1,M,Q}^{(\theta,-l,1,t)}\big)_{\nu_1}(t,x+W^{(\theta,l,1)}_t-W^{(\theta,l,1)}_s)\right|}
  {1+\|x+W^{(\theta, l, 1)}_{t}- W^{(\theta, l, 1)}_{s}\|_1^p}
  \cdot
  \tfrac
  {1+\|x+W^{(\theta, l, 1)}_{t}- W^{(\theta, l, 1)}_{s}\|_1^p}
  {1+\|x\|^p_1}
  \\
  &\qquad
  \cdot
  \left|
  \Big(
  1 ,
  \tfrac{ 
  W^{(\theta, l, 1)}_{t}- W^{(\theta, l, 1)}_{s}
  }{ t - s }
  \Big)_\nu
  \right|\Bigg]
  \\
  &+\sum_{t\in(s,T)}q^{Q,[s,T]}(t)
\E\!\left[ 
\sup_{x\in \R^d}
  \tfrac{\left| 
  (\funcF(0))
  (t,x+W^{(\theta, 0, 1)}_{t}- W^{(\theta, 0, 1)}_{s})\right|}
  {1+\|x+W^{(\theta, 0, 1)}_{t}- W^{(\theta, 0, 1)}_{s}\|_1^p}
  \cdot
  \tfrac
  {1+\|x+W^{(\theta, 0, 1)}_{t}- W^{(\theta, 0, 1)}_{s}\|_1^p}
  {1+\|x\|^p_1}
\left| 
  \Big(
  1 ,
  \tfrac{ 
  W^{(\theta, 0, 1)}_{t}- W^{(\theta, 0, 1)}_{s}
  }{ t - s }
  \Big)_\nu
  \right|\right]
\\
  &\leq
  \sup_{x\in \R^d} 
  \tfrac{\left|
    g(x)
    \right|}
    {1+\|x\|^p_1}
  +\sum_{\alpha=1}^dK_\alpha 
  \E\!\left[\left|
  (W^{0}_T-W^{0}_s)_\alpha
  \Big(
  1 ,
  \tfrac{ 
  W^{(\theta, 0, -i)}_{T}- W^{(\theta, 0, -i)}_{s}
  }{ T - s }
  \Big)_\nu
  \right|\right]
  \\
  &+\sum_{l=1}^{n}\sum_{t\in(s,T)}q^{Q,[s,T]}(t)
  \sum_{\nu_1=1}^{d+1}
  L_{\nu_1}
\E\!\left[\sup_{y\in \R^d}
\tfrac{\left| 
  \big({\bf U}_{l,M,Q}^{(\theta,l,1,t)}(t,y)- {\bf U}_{l-1,M,Q}^{(\theta,-l,1,t)}(t,y)\big)_{\nu_1}\right|
  (1+\|W^{(\theta, l, 1)}_{t}- W^{(\theta, l, 1)}_{s}\|_1)^p}
  {1+\|y\|^p_1}
  \left|
  \Big(
  1 ,
  \tfrac{ 
  W^{(\theta, l, 1)}_{t}- W^{(\theta, l, 1)}_{s}
  }{ t - s }
  \Big)_\nu
  \right|\right]
  \\
  &+\sum_{t\in(s,T)}q^{Q,[s,T]}(t)
  \left[\sup_{y\in \R^d}
  \tfrac{\left| 
  (\funcF(0))
  (t,y)\right|}
  {1+\|y\|^p_1}\right]
\E\!\left[
(1+\|W^{(\theta, 0, 1)}_{t}- W^{(\theta, 0, 1)}_{s}\|_1)^p
\left| 
  \Big(
  1 ,
  \tfrac{ 
  W^{(\theta, 0, 1)}_{t}- W^{(\theta, 0, 1)}_{s}
  }{ t - s }
  \Big)_\nu
  \right|\right]
  .
 \end{split}
 \end{equation} 
  The fact that for all $l\in \N$, $\theta\in \Theta$, $s,t\in [0,T)$ the random variables ${\bf U}_{l,M,Q}^{(\theta,l,1,t)}(t,\cdot)- {\bf U}_{l-1,M,Q}^{(\theta,-l,1,t)}(t,\cdot)$ and
  $W^{(\theta, l, 1)}_{t}- W^{(\theta, l, 1)}_{s}$ are independent proves that
  for all $\theta\in \Theta$, $\nu \in \{1,\ldots,d+1\}$, $l\in \N$, $s\in [0,T]$, $t\in (s,T]$ it holds that
  \begin{multline}\label{eq:uw_independent}
 \E\!\left[\sup_{x\in \R^d}
\tfrac{\left| 
  \big({\bf U}_{l,M,Q}^{(\theta,l,1,t)}(t,x)- {\bf U}_{l-1,M,Q}^{(\theta,-l,1,t)}(t,x)\big)_{\nu_1}\right|
  (1+\|W^{(\theta, l, 1)}_{t}- W^{(\theta, l, 1)}_{s}\|_1)^p}
  {1+\|x\|^p_1}
  \left|
  \Big(
  1 ,
  \tfrac{ 
  W^{(\theta, l, 1)}_{t}- W^{(\theta, l, 1)}_{s}
  }{ t - s }
  \Big)_\nu
  \right|\right]\\
  =\E\!\left[\sup_{x\in \R^d}
  \tfrac{\left| 
  \big({\bf U}_{l,M,Q}^{(\theta,l,1,t)}(t,x)- {\bf U}_{l-1,M,Q}^{(\theta,-l,1,t)}(t,x)\big)_{\nu_1}
   \right|}
   {1+\|x\|^p_1}
   \right]
   \E\!\left[(1+\|W^{(\theta, l, 1)}_{t}- W^{(\theta, l, 1)}_{s}\|_1)^p
   \left| 
  \Big(
  1 ,
  \tfrac{ 
  W^{(\theta, l, 1)}_{t}- W^{(\theta, l, 1)}_{s}
  }{ t - s }
  \Big)_\nu
  \right|\right].
  \end{multline}
  Combining \eqref{eq:def_gauss_leg}, \eqref{eq:approx_int1}, \eqref{eq:uw_independent}, the assumption \eqref{eq:approximations.integrable.assump}, and the induction hypothesis demonstrates that for all $\theta \in \Theta$, $s\in [0,T)$, $\nu \in \{1,\ldots,d+1\}$ it holds that
   \begin{equation}\label{eq:approx_int2}
  \E\!\left[\sup_{x\in \R^d}\frac{\left|\left({\bf U}_{n+1,M,Q}^{\theta}(s,x)\right)_\nu\right|}{1+\|x\|^p_1}\right]<\infty.
  \end{equation}
  This finishes the induction step $\N_0 \ni n \rightarrow n+1 \in \N$. Induction hence 
establishes~\eqref{item:approximations.integrable.i}.
Next we note that the triangle inequality and \eqref{eq:fLipschitz} imply that for all 
$\theta \in \Theta$, $n\in \N$, $s\in [0,T)$, $t\in (s,T]$, $x\in \R^d$, $\nu \in \{1,\ldots,d+1\}$ it holds that
\begin{equation}
\begin{split}
 &\E\left[\left|
  \big(\funcF({\bf U}_{n,M,Q}^{\theta})\big)
  (t,x+W_{t}^{0}-W_s^{0})
  \Big(
  1 ,
  \tfrac{ 
  W^{0}_{t}- W^{0}_{s}
  }{ t - s }
  \Big)_\nu\right|\right]\\
  &\le \E\left[\left|
  \big(\funcF(0)\big)
  (t,x+W_{t}^{0}-W_s^{0})
  \Big(
  1 ,
  \tfrac{ 
  W^{0}_{t}- W^{0}_{s}
  }{ t - s }
  \Big)_\nu\right|\right]
  +\sum_{\nu_1=1}^{d+1}L_{\nu_1} \E\left[\left|
  \big({\bf U}_{n,M,Q}^{\theta}\big)
  (t,x+W_{t}^{0}-W_s^{0})_{\nu_1}
  \Big(
  1 ,
  \tfrac{ 
  W^{0}_{t}- W^{0}_{s}
  }{ t - s }
  \Big)_\nu\right|\right]\\
  &\le 
\left(  
  \left[\sup_{y\in \R^d}
    \frac{\left| 
  \big(\funcF(0)\big)
  (t,y)\right|}
  {1+\|y\|^p_1}\right]
  +\sum_{\nu_1=1}^{d+1}L_{\nu_1} 
  \E\!\left[\sup_{y\in \R^d}\frac{\left|\left({\bf U}_{n,M,Q}^{\theta}(s,y)\right)_\nu\right|}{1+\|y\|^p_1}\right]
  \right)
  \E\left[\left(
  1+
  \|x+W_{t}^{0}-W_s^{0}\|_1^p
  \right)
  \left|
  \Big(
  1 ,
  \tfrac{ 
  W^{0}_{t}- W^{0}_{s}
  }{ t - s }
  \Big)_\nu\right|\right].
\end{split}
\end{equation}
This, \eqref{eq:approximations.integrable.assump}, and \eqref{item:approximations.integrable.i} prove \eqref{item:approximations.integrableiia}.
   Next we note that \eqref{eq:def:U}, \eqref{item:approximations.integrableiia}, the fact that 
  $({\bf U}_{n,M,Q}^{\theta})_{n\in\N_0}$, $\theta\in\Theta$, are identically
distributed, and a telescope argument yield that
  for all $n\in\N_0$, $\theta\in\Theta$, $s\in[0,T)$
  it holds $\P$-a.s.\ that
  \begin{equation}  \begin{split}
  &\E\left[{\bf U}_{n,M,Q}^{\theta}(s,x)\right]- \E\!\left[
    g(x+W_T^0-W_s^0)
     \Big(
     1 ,
     \tfrac{ 
     W^{0}_{T}- W^{0}_{s}
     }{ T - s }
     \Big)
    \right]
  \\
  &=\sum_{l=0}^{n-1}\sum_{t\in(s,T)}q^{Q,[s,T]}(t)
  \E\left[
  \big(\funcF({\bf U}_{l,M,Q}^{(\theta,l,0,t)})-\1_{\N}(l)\funcF( {\bf U}_{l-1,M,Q}^{(\theta,-l,0,t)})\big)
  (t,x+W_{t}^{(\theta,l,0)}-W_s^{(\theta,l,0)})
  \Big(
  1 ,
  \tfrac{ 
  W^{(\theta, l, 0)}_{t}- W^{(\theta, l, 0)}_{s}
  }{ t - s }
  \Big)\right]\\
& = \E\!\left[
    \sum_{t\in(s,T)}q^{Q,[s,T]}(t)
    \left(\funcF( {\bf U}_{n-1,M,Q}^{\theta})\right)\!(t,x+W_t^0-W_s^0)
     \Big(
     1 ,
     \tfrac{ 
     W^{0}_{t}- W^{0}_{s}
     }{ {t - s} }
     \Big)
    \right].
\end{split}     \end{equation}
  This establishes \eqref{item:approximations.integrable.ii}. The proof of Lemma \ref{l:approximations.integrable} is thus completed.
\end{proof}

\begin{lemma}[Nonlinear Feynman-Kac formula \& Bismut-Elworthy-Li formula]\label{l:nonlinear.FK.formula}
 Assume the setting in Section~\ref{sec:setting.full.discretization},
let $p\in \N$ and assume that
\begin{equation}  \begin{split}\label{eq:ldeltaintegrand3}
\sup_{(t,x)\in [0,T] \times \R^d} \frac{\|{\bf u}^\infty (t,x)\|_1}{1+\|x\|_1^p}+
\sup_{(t,x)\in [0,T] \times \R^d} \frac{|\funcF(0)(t,x)|}{1+\|x\|_1^p}<\infty.
\end{split}     \end{equation}
Then 
\begin{enumerate} [(i)] 
\item \label{item:approximations.integrable.iiFC}
  for all $s\in[0,T]$, $x\in\R^d$ it holds that
  \begin{equation}  \begin{split}\label{eq:FeynmanKac}
      u^{\infty}(s,x)-\E\!\left[g(x+W_{T-s}^0)\right]
      &=\E\!\left[
      \int_s^{T}(\funcF({\bf u}^{\infty}))(t,x+W_{t-s}^0)\,dt
      \right]
    \end{split}     \end{equation}
    and
    \item \label{item:approximations.integrable.iiiFC}
    for all $s\in[0,T)$, $x\in\R^d$ it holds
  that
  \begin{equation}  \begin{split}\label{eq:BEL}
      {\bf u}^{\infty}(s,x)-\E\!\left[g(x+W_T^0-W_s^0)
     \Big(
     1 ,
     \tfrac{ 
     W^{0}_{T}- W^{0}_{s}
     }{ T - s }
     \Big)
      \right]
      &=\E\!\left[
      \int_s^{T}(\funcF({\bf u}^{\infty}))(t,x+W_t^0-W_s^0)
     \Big(
     1 ,
     \tfrac{ 
     W^{0}_{t}- W^{0}_{s}
     }{ t - s }
     \Big)
      \,dt
      \right].
    \end{split}     \end{equation}
\end{enumerate}
\end{lemma}
\begin{proof}[Proof of Lemma~\ref{l:nonlinear.FK.formula}]
First note that the triangle inequality, \eqref{eq:fLipschitz}, and \eqref{eq:ldeltaintegrand3} ensure that
\begin{equation}\label{eq:ub_f.u.infty}
\sup_{(t,x)\in [0,T] \times \R^d} \frac{|\!\left( \funcF({\bf u}^{\infty}) \right)\!(t,x)|}{1+\|x\|_1^p}\le \sup_{(t,x)\in [0,T] \times \R^d}\frac{|\!\left( \funcF(0) \right)\!(t,x)|}{1+\|x\|_1^p} + \sup_{(t,x)\in [0,T] \times \R^d} \frac{\sum_{\nu=1}^{d+1}L_\nu | ({\bf u}^{\infty}(t,x))_\nu |}{1+\|x\|_1^p}<\infty.
\end{equation}
 It\^o's formula and the PDE~\eqref{eq:PDE} imply that
  for all $s\in[0,T]$, $t\in[s,T]$, $x\in \R^d$
  it holds $\P$-a.s.\ that
  \begin{align}
   \nonumber
   &
   u^{\infty}(t,x+W_t^0-W_s^0)-u^{\infty}(s,x)
   \\&=\int_s^t \left(
  \tfrac{\partial}{\partial r}u^{\infty}
+\tfrac{1}{2}\Delta_y u^{\infty}
  \right)\!(r,x+W_r^0-W_s^0)\,dr+\int_s^t\langle (\nabla_y u^{\infty})(r,x+W_r^0-W_s^0), \,dW_r^0\rangle 
  \label{eq:ItoFormula}
  \\&
   =-\int_s^t \left( \funcF({\bf u}^{\infty}) \right)\!(r,x+W_r^0-W_s^0)\,dr
   +\int_s^t\langle (\nabla_y u^{\infty})(r,x+W_r^0-W_s^0), \,dW_r^0\rangle .
   \nonumber
  \end{align}
 This, \eqref{eq:ldeltaintegrand3}, and \eqref{eq:ub_f.u.infty} show that for all $s\in[0,T]$, $x\in \R^d$ it holds that
  $\E\big[\sup_{t\in[s,T]}\big|
   \int_s^t\langle (\nabla_y u^{\infty})(r,x+W_r^0-W_s^0), \,dW_r^0\rangle \big|\big]<\infty$.
  This ensures that
  $\E\big[
   \int_s^T\langle (\nabla_y u^{\infty})(t,x+W_t^0-W_s^0), \,dW_t^0\rangle \big]=0$.
  This
  and \eqref{eq:ItoFormula} prove 
  for all $s\in[0,T]$, $x\in \R^d$
  that
  \begin{equation}  \begin{split}
      u^{\infty}(s,x)-\E[g(x+W_{T-s}^0)]
      =
      u^{\infty}(s,x)-\E[u^{\infty}(T,x+W_{T}^0-W_s^0)]
      =\E\!\left[
      \int_s^{T}(\funcF({\bf u}^{\infty}))(t,x+W_{t-s}^0)\,dt
      \right].
  \end{split}     \end{equation}
  This proves \eqref{item:approximations.integrable.iiFC}.
Next, the Bismut-Elworthy-Li formula (see, e.g., \cite[Proposition 3.2]{Fournie1999}) together with \eqref{eq:ldeltaintegrand3} show that 
  for all $i\in \{1,\ldots, d\}$, $s\in [0,T)$, $x\in \R^d$ it holds that
  \begin{equation}\label{eq:belg}
\tfrac {\partial}{\partial x_i} \E\!\left[g(x+W^0_{T-s})]
      \right] = \E\!\left[g(x+W_{T-s}^0)\tfrac{(W^0_{T-s})_i}{T-s}
      \right].
\end{equation}
Moreover, the Bismut-Elworthy-Li formula (see, e.g., \cite[Proposition 3.2]{Fournie1999}) together with \eqref{eq:ub_f.u.infty} demonstrate that 
  for all $i\in \{1,\ldots, d\}$, $s\in [0,T]$, $t\in (s,T]$, $x\in \R^d$ it holds that
\begin{equation}
\tfrac {\partial}{\partial x_i} \E\!\left[(\funcF({\bf u}^{\infty}))(t,x+W_{t-s}^0)
      \right] = \E\!\left[(\funcF({\bf u}^{\infty}))(t,x+W_{t-s}^0)\tfrac{(W^0_{t-s})_i}{t-s}
      \right].
\end{equation}
This and \eqref{eq:ub_f.u.infty} ensure that for all $i\in \{1,\ldots, d\}$, $s\in [0,T)$, $x\in \R^d$ it holds that
\begin{equation}
\tfrac {\partial}{\partial x_i} \int_s^T\E\!\left[(\funcF({\bf u}^{\infty}))(t,x+W_{t-s}^0)
      \right]\, dt  = \int_s^T\E\!\left[(\funcF({\bf u}^{\infty}))(t,x+W_{t-s}^0)\tfrac{(W^0_{t-s})_i}{t-s}
      \right]\, dt.
\end{equation}
Combining this, Fubini's theorem, \eqref{eq:FeynmanKac}, and \eqref{eq:belg} shows that 
for all $s\in[0,T)$, $x\in\R^d$ it holds
  that
  \begin{equation}  \begin{split}
      {\bf u}^{\infty}(s,x)-\E\!\left[g(x+W_T^0-W_s^0)
     \Big(
     1 ,
     \tfrac{ 
     W^{0}_{T}- W^{0}_{s}
     }{ T - s }
     \Big)
      \right]
      &=\E\!\left[
      \int_s^{T}(\funcF({\bf u}^{\infty}))(t,x+W_t^0-W_s^0)
     \Big(
     1 ,
     \tfrac{ 
     W^{0}_{t}- W^{0}_{s}
     }{ t - s }
     \Big)
      \,dt
      \right].
    \end{split}     \end{equation}
 This proves \eqref{item:approximations.integrable.iiiFC}. The proof of Lemma~\ref{l:nonlinear.FK.formula} is thus completed.
\end{proof}

\begin{lemma}[Recursive bound for global error]\label{l:estimate.L2error}
  Assume the setting in Section~\ref{sec:setting.full.discretization},
  let $p,M,Q\in \N$,
  assume that
\begin{equation}  \begin{split}\label{eq:ldeltaintegrand}
\sup_{(t,x)\in [0,T] \times \R^d} \frac{\|{\bf u}^\infty (t,x)\|_1}{1+\|x\|_1^p}+
\sup_{(t,x)\in [0,T] \times \R^d} \frac{|\funcF(0)(t,x)|}{1+\|x\|_1^p}<\infty,
\end{split}     \end{equation}
  and
   let $\eps \colon [0,T] \times \R^d \to [0,\infty]^{d+1}$ be the function that satisfies for all $s\in [0,T]$, $x\in \R^d$, $\nu \in \{1,\ldots, d+1\}$ that
\begin{equation}
\left(\eps(s,x)\right)_{\nu}=
\left|
      \E\!\left[
    \sum_{t\in(s,T)}q^{Q,[s,T]}(t)
    \left(\funcF({\bf u}^{\infty}\right)\!(t,x+W_{t-s}^0)
     \Big(
     1 ,
     \tfrac{ 
     W^{0}_{t-s}
          }{ {t - s} }
     \Big)_{\nu}
     -
      \smallint_s^{T}
      (\funcF({\bf u}^{\infty}))(t,x+W_{t-s}^0)
     \Big(
     1 ,
     \tfrac{ 
     W^{0}_{t-s}
     }{ t - s }
     \Big)_{\nu}
      \,dt
    \right]
      \right|.
\end{equation}
  Then
  for all $n, k\in\N$, $(t_0,x)\in[0,T)\times\R^d$, $\nu_0\in \{1,\ldots, d+1\}$
  it holds that
  \begin{equation}  \begin{split}\label{eq:estimate.L2error}
   &  \left\| \left( {\bf U}_{n,M,Q}^{0}(t_0,x)-{\bf u}^{\infty}(t_0,x)\right)_{\nu_0}\right\|_{L^2(\P;\R)}
   \\&
   \leq
   \sum_{j=0}^{k-1}
   \sum_{\substack{l_1,\ldots,l_{j+1}\in\N,\\ l_1<\ldots<l_{j+1}=n}}
   \sum_{\substack{t_1,\ldots,t_j,t_{j+1}\in\R,\\t_0<t_1<\ldots<t_j<t_{j+1}\leq T}}
    \sum_{\nu_1,\ldots,\nu_{j+1} \in \{1,\ldots,d+1\}}
   \tfrac{2^j}{\sqrt{M^{n-j-l_1}}}
   \left[\prod_{i=1}^j L_{\nu_i}q^{Q,[t_{i-1},T]}(t_i)\right]
  \\&\quad\cdot
   \Bigg\{
   \1_{\{1\}}(\nu_{j+1})
   \Bigg(
    \1_{\{T\}}(t_{j+1})\Bigg(\left\|\left(\eps(t_j,x+W^0_{t_j}-W^0_{t_0})\right)_{\nu_j}\prod_{i=1}^j
\Big(   
   1,
     \tfrac{ 
     W^{0}_{t_i}- W^{0}_{t_{i-1}}
     }{ {t_{i}-t_{i-1}} } 
    \Big)_{\nu_{i-1}}
    \right\|_{L^2(\P;\R)}
    \\&\qquad \quad
    +\tfrac{1}{\sqrt{M^{l_1}}}
   \left\|\left(g(x+W^0_{T}-W^0_{t_0})-g(x+W^0_{t_j}-W^0_{t_0})\right)
   \prod_{i=1}^{j+1}
     \Big(   
   1,
     \tfrac{ 
     W^{0}_{t_i}- W^{0}_{t_{i-1}}
     }{ {t_{i}-t_{i-1}} } 
     \Big)_{\nu_{i-1}}
   \right\|_{L^2(\P;\R)}
   \Bigg)
   \\&\quad\qquad
   +
   \tfrac{q^{Q,[t_j,T]}(t_{j+1})}{\sqrt{M^{l_1}}}
   \left\|\left(F(0)\right)\!(t_{j+1},x+W_{t_{j+1}}^0-W_{t_0}^0)
   \prod_{i=1}^{j+1}
     \Big(   
   1,
     \tfrac{ 
     W^{0}_{t_i}- W^{0}_{t_{i-1}}
     }{ {t_{i}-t_{i-1}} } 
     \Big)
     _{\nu_{i-1}}
   \right\|_{L^2(\P;\R)}\Bigg)
   \\&\quad\qquad
   +
   \tfrac{L_{\nu_{j+1}}q^{Q,[t_j,T]}(t_{j+1})}{\sqrt{M^{l_1-1}}}
   \left\|\left({\bf u}^{\infty}(t_{j+1},x+W_{t_{j+1}}^0-W_{t_0}^0)\right)_{\nu_{j+1}}
   \prod_{i=1}^{j+1}
     \Big(   
   1,
     \tfrac{ 
     W^{0}_{t_i}- W^{0}_{t_{i-1}}
     }{ {t_{i}-t_{i-1}} } 
     \Big)_{\nu_{i-1}}
   \right\|_{L^2(\P;\R)}
   \Bigg\}
   \\&
   +
   \sum_{\substack{l_1,\ldots,l_{k}\in\N,\\ l_1<\ldots<l_k<n}}
   \;\;
   \sum_{\substack{t_1,\ldots,t_k\in\R,\\t_0<t_1<\ldots<t_k<T}}
   \sum_{\nu_1,\ldots,\nu_k \in \{1,\ldots,d+1\}}
   \tfrac{2^k}{\sqrt{M^{n-k-l_1}}}\left[\prod_{i=1}^{k}L_{\nu_i}q^{Q,[t_{i-1},T]}(t_i)   \right]
  \\&\quad\cdot
   \left\|\left(
   \left({\bf U}_{l_1,M,Q}^{0}-{\bf u}^{\infty}\right)(t_k,x+W_{t_k}^0-W_{t_0}^0)\right)_{\nu_k}
   \prod_{i=1}^{k}
     \Big(   
   1,
     \tfrac{ 
     W^{0}_{t_i}- W^{0}_{t_{i-1}}
     }{ {t_{i}-t_{i-1}} } 
     \Big)_{\nu_{i-1}}
   \right\|_{L^2(\P;\R)}.
  \end{split}     \end{equation}
\end{lemma}
\begin{proof}[Proof of Lemma~\ref{l:estimate.L2error}]
We note that \eqref{eq:ldeltaintegrand} and \eqref{eq:fLipschitz} ensure that the function $\eps$ is well-defined.
 First, we analyze the \emph{Monte Carlo error}.
 Independence, Items~\eqref{item:approximations.integrable.i} and \eqref{item:approximations.integrableiia} of Lemma~\ref{l:approximations.integrable}, and~\eqref{eq:def:U} imply that
 for all  $m\in\N$, $x\in\R^d$, $s\in[0,T)$, $\nu \in \{1,\ldots,d+1\}$ it holds that
  \begin{equation} \label{eq:var_approx1} \begin{split}
    &\Var\!\left(
       \left({\bf U}_{m,M,Q}^{0}(s,x)\right)_{\nu}
     \right)
     =\tfrac{1}{M^m}
     \Var\!\left(\left(g(x+W_T^{0}-W_s^0)-g(x)\right)
     \Big(
     1 ,
     \tfrac{ 
     W^{0}_{T}- W^{0}_{s}
     }{ T - s }
     \Big)_{\nu}
    \right)
     \\&+\sum_{l=0}^{m-1}\tfrac{1}{M^{m-l}}
     \Var\!\left(\sum_{t\in(s,T)} q^{Q,[s,T]}(t)\left(\funcF( {\bf U}_{l,M,Q}^{(0,l,1,t)})-\1_{\N}(l)\funcF( {\bf U}_{l-1,M,Q}^{(0,-l,1,t)})\right)\!
     (t,x+W_t^0-W_s^0)\,
     \Big(
     1 ,
     \tfrac{ 
     W^{0}_{t}- W^{0}_{s}
     }{ {t - s} }
     \Big)_{\nu}
     \right)
   \\&
   \leq\tfrac{1}{M^m}
     \E\!\left[\left|
     \left(g(x+W_T^{0}-W_s^0)-g(x)\right)
     \Big(
     1 ,
     \tfrac{ 
     W^{0}_{T}- W^{0}_{s}
     }{ T - s }
     \Big)_{\nu}
    \right|^2\right]
     \\&+\sum_{l=0}^{m-1}\tfrac{1}{M^{m-l}}
     \E\!\left[\left|\sum_{t\in(s,T)} q^{Q,[s,T]}(t)\left(\funcF( {\bf U}_{l,M,Q}^{(0,l,1,t)})-\1_{\N}(l)\funcF( {\bf U}_{l-1,M,Q}^{(0,-l,1,t)})\right)\!
     (t,x+W_t^0-W_s^0)\,
     \Big(
     1 ,
     \tfrac{ 
     W^{0}_{t}- W^{0}_{s}
     }{ {t - s} }
     \Big)_{\nu}
     \right|^2\right].
  \end{split}     \end{equation}
 Combining this, the triangle inequality, and \eqref{eq:fLipschitz} 
 yields that
  for all 
  $m\in\N$, $x\in\R^d$, $s\in[0,T)$, $\nu \in \{1,\ldots,d+1\}$
  it holds that
  \begin{equation}  \begin{split}
   &     \left\| \left({\bf U}_{m,M,Q}^{0}(s,x)-\E\!\left[{\bf U}_{m,M,Q}^{0}(s,x)\right]\right)_{\nu}\right\|_{L^2(\P;\R)}
   =
   \left(
      \Var\!\left(
      \left(
       {\bf U}_{m,M,Q}^{0}(s,x)
       \right)_{\nu}
     \right)
     \right)^{\nicefrac{1}{2}}
 \\&
    \leq
     \tfrac{1}{\sqrt{M^m}}
     \left\|
     \left(g(x+W_T^{0}-W_s^0)-g(x)\right)
     \Big(
     1 ,
     \tfrac{ 
     W^{0}_{T}- W^{0}_{s}
     }{ T - s }
     \Big)_{\nu}
     \right\|_{L^2(\P;\R)}
    \\&
     +
    \sum_{l=0}^{m-1}
    \left[
     \sum_{t\in(s,T)} \tfrac{q^{Q,[s,T]}(t)}{\sqrt{M^{m-l}}}
    \left\|
     \left(\funcF( {\bf U}_{l,M,Q}^{(0,l,1,t)})-\1_{\N}(l)\funcF( {\bf U}_{l-1,M,Q}^{(0,-l,1,t)})\right)\!
     (t,x+W_t^0-W_s^0)
     \Big(
     1 ,
     \tfrac{ 
     W^{0}_{t}- W^{0}_{s}
     }{ {t - s} }
     \Big)_{\nu}
    \right\|_{L^2(\P;\R)}
    \right]
 \\&
    \leq
     \tfrac{1}{\sqrt{M^m}}
     \left\|
     \left(g(x+W_T^{0}-W_s^0)-g(x)\right)
     \Big(
     1 ,
     \tfrac{ 
     W^{0}_{T}- W^{0}_{s}
     }{ T - s }
     \Big)_{\nu}
     \right\|_{L^2(\P;\R)}
  \\&
    +
     \tfrac{1}{\sqrt{M^m}}
     \sum_{t\in(s,T)} q^{Q,[s,T]}(t)
    \left\|
     \left(\funcF( 0)\right)\!
     (t,x+W_t^0-W_s^0)
     \Big(
     1 ,
     \tfrac{ 
     W^{0}_{t}- W^{0}_{s}
     }{ {t - s} }
     \Big)_{\nu}
    \right\|_{L^2(\P;\R)}
    \\&
     +
    \sum_{l=1}^{m-1}
    \left[
     \sum_{t\in(s,T)} \tfrac{q^{Q,[s,T]}(t)}{\sqrt{M^{m-l}}}
    \left\|
    \sum_{\nu_1=1}^{d+1}L_{\nu_1}
     \left|
	\left(     
     \left( {\bf U}_{l,M,Q}^{(0,l,1,t)}-{\bf U}_{l-1,M,Q}^{(0,-l,1,t)}\right)\!
     (t,x+W_t^0-W_s^0)
   \right)_{\nu_1}  
     \right|
     \left|
     \Big(
     1 ,
     \tfrac{ 
     W^{0}_{t}- W^{0}_{s}
     }{ {t - s} }
     \Big)_{\nu}
     \right|
    \right\|_{L^2(\P;\R)}
    \right].
    \end{split}
    \end{equation}
    This and the triangle inequality ensure that
  for all 
  $m\in\N$, $x\in\R^d$, $s\in[0,T)$, $\nu \in \{1,\ldots,d+1\}$
  it holds that
    \begin{equation}\label{eq:final.MCerror}
    \begin{split}
     &     \left\| \left({\bf U}_{m,M,Q}^{0}(s,x)-\E\!\left[{\bf U}_{m,M,Q}^{0}(s,x)\right]\right)_{\nu}\right\|_{L^2(\P;\R)}\\
    &
    \leq
     \tfrac{1}{\sqrt{M^m}}
     \left\|
     \left(g(x+W_T^{0}-W_s^0)-g(x)\right)
     \Big(
     1 ,
     \tfrac{ 
     W^{0}_{T}- W^{0}_{s}
     }{ T - s }
     \Big)_{\nu}
     \right\|_{L^2(\P;\R)}
  \\&
    +
     \tfrac{1}{\sqrt{M^m}}
     \sum_{t\in(s,T)} q^{Q,[s,T]}(t)
    \left\|
     \left(\funcF( 0)\right)\!
     (t,x+W_t^0-W_s^0)
     \Big(
     1 ,
     \tfrac{ 
     W^{0}_{t}- W^{0}_{s}
     }{ {t - s} }
     \Big)_{\nu}
    \right\|_{L^2(\P;\R)}
    \\&
     +
    \sum_{l=1}^{m-1}
     \sum_{t\in(s,T)} 
     \sum_{\nu_1=1}^{d+1}
     \tfrac{L_{\nu_1}q^{Q,[s,T]}(t)}{\sqrt{M^{m-l}}}
     \left\|\left(\left( {\bf U}_{l,M,Q}^{0}-{\bf u}^{\infty}\right)\!
     (t,x+W_t^0-W_s^0)\right)_{\nu_1}
     \Big(
     1 ,
     \tfrac{ 
     W^{0}_{t}- W^{0}_{s}
     }{ {t - s} }
     \Big)_{\nu}
    \right\|_{L^2(\P;\R)}
      \\&
     +
    \sum_{l=1}^{m-1}
     \sum_{t\in(s,T)}  
     \sum_{\nu_1=1}^{d+1}
     \tfrac{L_{\nu_1}q^{Q,[s,T]}(t)}{\sqrt{M^{m-l}}}
    \left\|
     \left(\left( {\bf U}_{l-1,M,Q}^{0}-{\bf u}^{\infty}\right)\!
     (t,x+W_t^0-W_s^0)\right)_{\nu_1}
     \Big(
     1 ,
     \tfrac{ 
     W^{0}_{t}- W^{0}_{s}
     }{ {t - s} }
     \Big)_{\nu}
    \right\|_{L^2(\P;\R)}
    \\&
    =
     \tfrac{1}{\sqrt{M^m}}
     \left\|
     \left(g(x+W_T^{0}-W_s^0)-g(x)\right)
     \Big(
     1 ,
     \tfrac{ 
     W^{0}_{T}- W^{0}_{s}
     }{ T - s }
     \Big)_{\nu}
     \right\|_{L^2(\P;\R)}
  \\&
    +
     \tfrac{1}{\sqrt{M^m}}
     \sum_{t\in(s,T)} q^{Q,[s,T]}(t)
    \left\|
     \left(\funcF( 0)\right)\!
     (t,x+W_t^0-W_s^0)
     \Big(
     1 ,
     \tfrac{ 
     W^{0}_{t}- W^{0}_{s}
     }{ {t - s} }
     \Big)_{\nu}
    \right\|_{L^2(\P;\R)} 
    \\&
     +
    \sum_{l=0}^{m-1}
    \sum_{t\in(s,T)} 
    \sum_{\nu_1=1}^{d+1} \tfrac{ L_{\nu_1}q^{Q,[s,T]}(t)}{\sqrt{M^{m-l-1}}}
    \left(
    \tfrac{\1_{(0,m)}(l)}{\sqrt{M}}+\1_{(-1,m-1)}(l)
    \right)
    \left\|
     \left(\left( {\bf U}_{l,M,Q}^{0}-{\bf u}^{\infty}\right)\!
     (t,x+W_t^0-W_s^0)\right)_{\nu_1}
     \Big(
     1 ,
     \tfrac{ 
     W^{0}_{t}- W^{0}_{s}
     }{ {t - s} }
     \Big)_{\nu}
    \right\|_{L^2(\P;\R)}
    .
  \end{split}     \end{equation}
  Next we analyze the \emph{time discretization error}.
Item~\eqref{item:approximations.integrable.ii} of Lemma~\ref{l:approximations.integrable} ensures that
  for all  $m\in\N$, $s\in[0,T)$, $x\in\R^d$
   it holds that
  \begin{equation}  \begin{split}\label{eq:discreteFeynmanKac2}
    &\E\!\left[{\bf U}_{m,M,Q}^{0}(s,x)
    -g(x+W_T^0-W_s^0)
     \Big(
     1 ,
     \tfrac{ 
     W^{0}_{T}- W^{0}_{s}
     }{ T - s }
     \Big)
    \right]
    \\&=
    \E\!\left[
    \sum_{t\in(s,T)}q^{Q,[s,T]}(t)
    \left(\funcF( {\bf U}_{m-1,M,Q}^{0})\right)\!(t,x+W_t^0-W_s^0)
     \Big(
     1 ,
     \tfrac{ 
     W^{0}_{t}- W^{0}_{s}
     }{ {t - s} }
     \Big)
    \right].
  \end{split}     \end{equation}
  Item \eqref{item:approximations.integrable.iiiFC} of Lemma \ref{l:nonlinear.FK.formula} proves that 
  for all $s\in[0,T)$, $x\in\R^d$ it holds
  that
  \begin{equation}  \begin{split}\label{eq:BEL2}
      {\bf u}^{\infty}(s,x)-\E\!\left[g(x+W_T^0-W_s^0)
     \Big(
     1 ,
     \tfrac{ 
     W^{0}_{T}- W^{0}_{s}
     }{ T - s }
     \Big)
      \right]
      &=\E\!\left[
      \int_s^{T}(\funcF({\bf u}^{\infty}))(t,x+W_t^0-W_s^0)
     \Big(
     1 ,
     \tfrac{ 
     W^{0}_{t}- W^{0}_{s}
     }{ t - s }
     \Big)
      \,dt
      \right].
    \end{split}     \end{equation}
This,~\eqref{eq:discreteFeynmanKac2}, 
the triangle inequality,
and Jensen's inequality
show
  for all  $m\in\N$, $s\in[0,T)$, $x\in\R^d$, $\nu \in \{1,\ldots,d+1\}$
  that
  \begin{equation}  \begin{split}\label{eq:final.Timeerror}
    &\left|\left(\E\!\left[{\bf U}_{m,M,Q}^{0}(s,x)\right]
            -{\bf u}^{\infty}(s,x)\right)_{\nu}\right|
  \\&=
    \Bigg|
    \E\!\left[
    \sum_{t\in(s,T)}q^{Q,[s,T]}(t)
    \left(\funcF( {\bf U}_{m-1,M,Q}^{0})\right)\!(t,x+W_t^0-W_s^0)
     \Big(
     1 ,
     \tfrac{ 
     W^{0}_{t}- W^{0}_{s}
     }{ {t - s} }
     \Big)_{\nu}
     \right]
     \\&\qquad
      -
      \E\!\left[
      \smallint_s^{T}
      (\funcF({\bf u}^{\infty}))(t,x+W_t^0-W_s^0)
     \Big(
     1 ,
     \tfrac{ 
     W^{0}_{t}- W^{0}_{s}
     }{ t - s }
     \Big)_{\nu}
      \,dt
    \right]
      \Bigg|
  \\&\leq
    \left|
    \E\!\left[
    \sum_{t\in(s,T)}q^{Q,[s,T]}(t)
    \left(\funcF( {\bf U}_{m-1,M,Q}^{0})-\funcF({\bf u}^{\infty})\right)\!(t,x+W_t^0-W_s^0)
     \Big(
     1 ,
     \tfrac{ 
     W^{0}_{t}- W^{0}_{s}
     }{ {t - s} }
     \Big)_{\nu}
     \right]
     \right|
     \\&\qquad
      +
      \left|
      \E\!\left[
    \sum_{t\in(s,T)}q^{Q,[s,T]}(t)
    \left(\funcF({\bf u}^{\infty})\right)\!(t,x+W_{t-s}^0)
     \Big(
     1 ,
     \tfrac{ 
     W^{0}_{t-s}
     }{ {t - s} }
     \Big)_{\nu}
     -
      \smallint_s^{T}
      (\funcF({\bf u}^{\infty}))(t,x+W_{t-s}^0)
     \Big(
     1 ,
     \tfrac{ 
     W^{0}_{t-s}
     }{ t - s }
     \Big)_{\nu}
      \,dt
    \right]
      \right|
  \\&=
    \left|
    \E\!\left[
    \sum_{t\in(s,T)}q^{Q,[s,T]}(t)
    \left(\funcF( {\bf U}_{m-1,M,Q}^{0})-\funcF({\bf u}^{\infty})\right)\!(t,x+W_t^0-W_s^0)
     \Big(
     1 ,
     \tfrac{ 
     W^{0}_{t}- W^{0}_{s}
     }{ {t - s} }
     \Big)_{\nu}
     \right]
     \right|
     +\left(\eps(s,x)\right)_{\nu}
 \\&\leq 
    \E\!\left[
    \sum_{t\in(s,T)}q^{Q,[s,T]}(t)
    \left[
    \sum_{\nu_1=1}^{d+1}L_{\nu_1}
    \left| 
   \left( 
    \left({\bf U}_{m-1,M,Q}^{0}-{\bf u}^{\infty}\right)\!(t,x+W_t^0-W_s^0)
   \right)_{\nu_1} 
    \right|
    \right]
     \left|\Big(
     1 ,
     \tfrac{ 
     W^{0}_{t}- W^{0}_{s}
     }{ {t - s} }
     \Big)_{\nu}\right|
     \right]
     +\left(\eps(s,x)\right)_{\nu}
     \\&\leq 
    \sum_{t\in(s,T)}\sum_{\nu_1=1}^{d+1}
    L_{\nu_1}q^{Q,[s,T]}(t)
    \left\|
    \left( \left({\bf U}_{m-1,M,Q}^{0}-{\bf u}^{\infty}\right)\!(t,x+W_t^0-W_s^0)\right)_{\nu_1}
    \Big(
     1 ,
     \tfrac{ 
     W^{0}_{t}- W^{0}_{s}
     }{ {t - s} }
     \Big)_{\nu}
     \right\|_{L^2(\P;\R)}
     +\left(\eps(s,x)\right)_{\nu}.
  \end{split}     \end{equation}
  In the next step we combine the established bounds for the Monte Carlo error and for the time discretization error
  to obtain a bound for the
  \emph{global error}. More formally, observe that 
  \eqref{eq:final.MCerror} and \eqref{eq:final.Timeerror} ensure that 
  for all  $m\in\N$, $s\in[0,T)$, $x\in\R^d$, $\nu \in \{1,\ldots, d+1\}$
  it holds that
  \begin{equation}  \begin{split}\label{eq:global.estimate}
   &     \left\| \left({\bf U}_{m,M,Q}^{0}(s,x)-{\bf u}^{\infty}(s,x)\right)_{\nu}\right\|_{L^2(\P;\R)}\\
   & \leq \left\| \left({\bf U}_{m,M,Q}^{0}(s,x)-\E\!\left[{\bf U}_{m,M,Q}^{0}(s,x)\right]\right)_{\nu}\right\|_{L^2(\P;\R)} 
   +\left|\left(\E\!\left[{\bf U}_{m,M,Q}^{0}(s,x)\right]
            -{\bf u}^{\infty}(s,x)\right)_{\nu}\right|
\\&
    \leq
     \tfrac{1}{\sqrt{M^m}}
     \left\|
     \left(g(x+W_T^{0}-W_s^0)-g(x)\right)
     \Big(
     1 ,
     \tfrac{ 
     W^{0}_{T}- W^{0}_{s}
     }{ T - s }
     \Big)_{\nu}
     \right\|_{L^2(\P;\R)}
  \\&
    +
     \tfrac{1}{\sqrt{M^m}}
     \sum_{t\in(s,T)} q^{Q,[s,T]}(t)
    \left\|
     \left(\funcF( 0)\right)\!
     (t,x+W_t^0-W_s^0)
     \Big(
     1 ,
     \tfrac{ 
     W^{0}_{t}- W^{0}_{s}
     }{ {t - s} }
     \Big)_{\nu}
    \right\|_{L^2(\P;\R)} +\left(\eps(s,x)\right)_{\nu}
    \\&
    + \sum_{t\in(s,T)}\sum_{\nu_1=1}^{d+1}
    L_{\nu_1}q^{Q,[s,T]}(t)
    \left\|
    \left( \left({\bf U}_{m-1,M,Q}^{0}-{\bf u}^{\infty}\right)\!(t,x+W_t^0-W_s^0)\right)_{\nu_1}
    \Big(
     1 ,
     \tfrac{ 
     W^{0}_{t}- W^{0}_{s}
     }{ {t - s} }
     \Big)_{\nu}
     \right\|_{L^2(\P;\R)}
     \\&
     +\sum_{l=0}^{m-1}
    \sum_{t\in(s,T)} 
    \sum_{\nu_1=1}^{d+1} \tfrac{ L_{\nu_1}q^{Q,[s,T]}(t)}{\sqrt{M^{m-l-1}}}
    \left(
    \tfrac{\1_{(0,m)}(l)}{\sqrt{M}}+\1_{(-1,m-1)}(l)
    \right)
    \left\|
     \left(\left( {\bf U}_{l,M,Q}^{0}-{\bf u}^{\infty}\right)\!
     (t,x+W_t^0-W_s^0)\right)_{\nu_1}
     \Big(
     1 ,
     \tfrac{ 
     W^{0}_{t}- W^{0}_{s}
     }{ {t - s} }
     \Big)_{\nu}
    \right\|_{L^2(\P;\R)}   
   \\
  &\leq
  \left(\eps(s,x)\right)_{\nu}+
     \tfrac{1}{\sqrt{M^m}}
     \left\|
     \left(g(x+W_T^{0}-W_s^0)-g(x)\right)
     \Big(
     1 ,
     \tfrac{ 
     W^{0}_{T}- W^{0}_{s}
     }{ T - s }
     \Big)_{\nu}
     \right\|_{L^2(\P;\R)}
  \\&
    +
     \tfrac{1}{\sqrt{M^m}}
     \sum_{t\in(s,T)} q^{Q,[s,T]}(t)
    \left\|
     \left(\funcF( 0)\right)\!
     (t,x+W_t^0-W_s^0)
     \Big(
     1 ,
     \tfrac{ 
     W^{0}_{t}- W^{0}_{s}
     }{ {t - s} }
     \Big)_{\nu}
    \right\|_{L^2(\P;\R)} 
    \\&
    +
     \sum_{t\in(s,T)}  \sum_{\nu_1=1}^{d+1} 
     \tfrac{L_{\nu_1} q^{Q,[s,T]}(t)}{\sqrt{M^{m-1}}}
    \left\|
     \left({\bf u}^{\infty}
     (t,x+W_t^0-W_s^0)\right)_{\nu_1}
     \Big(
     1 ,
     \tfrac{ 
     W^{0}_{t}- W^{0}_{s}
     }{ {t - s} }
     \Big)_{\nu}
    \right\|_{L^2(\P;\R)}
   \\&
    +\sum_{l=1}^{m-1}
    \sum_{t\in(s,T)} 
    \sum_{\nu_1=1}^{d+1} \tfrac{2 L_{\nu_1}q^{Q,[s,T]}(t)}{\sqrt{M^{m-l-1}}}
    \left\|
     \left(\left( {\bf U}_{l,M,Q}^{0}-{\bf u}^{\infty}\right)\!
     (t,x+W_t^0-W_s^0)\right)_{\nu_1}
     \Big(
     1 ,
     \tfrac{ 
     W^{0}_{t}- W^{0}_{s}
     }{ {t - s} }
     \Big)_{\nu}
    \right\|_{L^2(\P;\R)} .
  \end{split}     \end{equation}
  We prove~\eqref{eq:estimate.L2error} by induction on $k\in\N$.
  The base case $k=1$ follows immediately from~\eqref{eq:global.estimate}.
  For the induction step $\N\ni k\mapsto k+1\in\N$ let $k\in\N$ and assume that~\eqref{eq:estimate.L2error}
  holds for $k$.
  Inequality~\eqref{eq:global.estimate}
  and independence of $({\bf U}^0_{m,M,Q})_{m\in\N_0}$ and $W^0$
  yield that
  for all  $m\in\N$, $t_1,\ldots,t_{k}\in\R$, $x\in\R^d$, $\nu_0,\ldots, \nu_{k} \in \{1,\ldots, d+1\}$ with $ t_0<t_1<\ldots<t_k<T$
  it holds that
  \begin{equation}  \begin{split}
   &\left\|\left(\left({\bf U}_{{l_{1}},M,Q}^{0}-{\bf u}^{\infty}\right)(t_k,x+W_{t_k}^0-W_{t_0}^0)\right)_{\nu_k}\prod_{i=1}^{k}
   \Big(   
   1,
     \tfrac{ 
     W^{0}_{t_i}- W^{0}_{t_{i-1}}
     }{ {t_{i}-t_{i-1}} } 
     \Big)_{\nu_{i-1}}
   \right\|_{L^2(\P;\R)}
   \\&
   =
   \left(
   \E\!\left[
   \left(
   \left\|\left({\bf U}_{{l_{1}},M,Q}^{0}-{\bf u}^{\infty}\right)(t_k,z)
   \right)_{\nu_k}
   \right\|_{L^2(\P;\R)}^2\Big|_{z=x+W_{t_k}^0-W^0_{t_0}}
   \prod_{i=1}^{k}
     \Big(\Big(   
   1,
     \tfrac{ 
     W^{0}_{t_i}- W^{0}_{t_{i-1}}
     }{ {t_{i}-t_{i-1}} } 
     \Big)_{\nu_{i-1}}\Big)
     ^2
   \right]
   \right)^{\frac{1}{2}}
   \\&
   \leq
   \left\|
   \left(  
     \eps(t_k,x+W_{t_k}^0-W^0_{t_0})
     \right)_{\nu_k}
   \prod_{i=1}^{k}
     \Big(   
   1,
     \tfrac{ 
     W^{0}_{t_i}- W^{0}_{t_{i-1}}
     }{ {t_{i}-t_{i-1}} } 
     \Big)
     _{\nu_{i-1}}
   \right\|_{L^2(\P;\R)}
   \\&\quad 
   +\tfrac{1}{\sqrt{M^{l_{1}}}}
   \left\|\left(g(x+W^0_{T}-W^0_{t_0})-g(x+W^0_{t_k}-W^0_{t_0})\right)
     \Big(   
   1 ,
     \tfrac{ 
     W^{0}_{T}- W^{0}_{t_{k}}
     }{ T-t_k } 
     \Big)
     _{\nu_{k}}
   \prod_{i=1}^{k}
     \Big(   
   1,
     \tfrac{ 
     W^{0}_{t_i}- W^{0}_{t_{i-1}}
     }{ {t_{i}-t_{i-1}} } 
     \Big)
     _{\nu_{i-1}}
   \right\|_{L^2(\P;\R)}
   \\&\quad
   +
   \tfrac{1}{\sqrt{M^{l_{1}}}}
   \sum_{t_{k+1}\in(t_k,T)}q^{Q,[t_k,T]}(t_{k+1})
   \left\|\left(\funcF(0)\right)(t_{k+1},x+W_{t_{k+1}}^0-W_{t_0}^0)
   \prod_{i=1}^{k+1}
     \Big(   
   1,
     \tfrac{ 
     W^{0}_{t_i}- W^{0}_{t_{i-1}}
     }{ {t_{i}-t_{i-1}} } 
     \Big)
     _{\nu_{i-1}}
   \right\|_{L^2(\P;\R)}
   \\&\quad
   +
   \sum_{t_{k+1}\in(t_k,T)}
   \sum_{\nu_{k+1}=1}^{d+1}
   \tfrac{L_{\nu_{k+1}}q^{Q,[t_k,T]}(t_{k+1})}{\sqrt{M^{{l_{1}}-1}}}
   \left\|
  \left( 
   {\bf u}^{\infty}(t_{k+1},x+W_{t_{k+1}}^0-W_{t_0}^0)
   \right)_{\nu_{k+1}}
   \prod_{i=1}^{k+1}
     \Big(   
   1,
     \tfrac{ 
     W^{0}_{t_i}- W^{0}_{t_{i-1}}
     }{ {t_{i}-t_{i-1}} } 
     \Big)
     _{\nu_{i-1}}
   \right\|_{L^2(\P;\R)}
   \\&\quad 
   +
   \sum_{{l_{0}}=1}^{m-1}
   \;\;
   \sum_{t_{k+1}\in(t_k,T)}
   \sum_{\nu_{k+1}=1}^{d+1}
   \tfrac{2L_{\nu_{k+1}}q^{Q,[t_{k},T]}(t_{k+1})}{\sqrt{M^{{l_{1}}-1-{l_{0}}}}}
  \\&\quad\quad\cdot
   \left\|
   \left(
   \left({\bf U}_{{l_{0}},M,Q}^0-{\bf u}^{\infty}\right)(t_{k+1},x+W_{t_{k+1}}^0-W_{t_0}^0)
   \right)_{\nu_{k+1}}
   \prod_{i=1}^{k+1}
     \Big(   
   1,
     \tfrac{ 
     W^{0}_{t_i}- W^{0}_{t_{i-1}}
     }{ {t_{i}-t_{i-1}} } 
     \Big)
     _{\nu_{i-1}}
   \right\|_{L^2(\P;\R)}.  
  \end{split}     \end{equation}
This and the induction hypothesis
complete the induction step $\N \ni k \rightarrow k+1 \in \N$. Induction hence 
establishes~\eqref{eq:estimate.L2error}.
  This finishes the proof of Lemma \ref{l:estimate.L2error}.
\end{proof}
\begin{theorem}[Global approximation error]\label{thm:rate}
Assume the setting in Section~\ref{sec:setting.full.discretization},
let $p,n,Q\in \N$, $M\in \N\cap [2,\infty)$, $\nu_0\in \{1,\ldots,d+1\}$, $(t_0,x)\in [0,T)\times \R^d$, 
assume that
\begin{equation}  \begin{split}\label{eq:ldeltaintegrand2}
\sup_{(t,z)\in [0,T] \times \R^d} \frac{\|{\bf u}^\infty (t,z)\|_1}{1+\|z\|_1^p}+
\sup_{(t,z)\in [0,T] \times \R^d} \frac{|\funcF(0)(t,z)|}{1+\|z\|_1^p}<\infty,
\end{split}     \end{equation}
let $C\in [0,\infty)$ be the real number given by
\begin{equation}\label{eq:defC1}
C=2(\sqrt{T-t_0}+1)\sqrt{(T-t_0)\pi}\left(\|L\|_1+1\right)+1,
\end{equation}
and let $\eps \colon [0,T] \times \R^d \to [0,\infty]^{d+1}$ be the function that satisfies for all $s\in [0,T]$, $y\in \R^d$, $\nu \in \{1,\ldots, d+1\}$ that
\begin{equation}
\left(\eps(s,y)\right)_{\nu}=
\left|
      \E\!\left[
    \sum_{t\in(s,T)}q^{Q,[s,T]}(t)
    \left(\funcF({\bf u}^{\infty}\right)\!(t,y+W_{t-s}^0)
     \Big(
     1 ,
     \tfrac{ 
     W^{0}_{t-s}
          }{ {t - s} }
     \Big)_{\nu}
     -
      \smallint_s^{T}
      (\funcF({\bf u}^{\infty}))(t,y+W_{t-s}^0)
     \Big(
     1 ,
     \tfrac{ 
     W^{0}_{t-s}
     }{ t - s }
     \Big)_{\nu}
      \,dt
    \right]
      \right|.
\end{equation}
 Then it holds that
 \begin{equation} \begin{split}
    & \left\| \left( {\bf U}_{n,M,Q}^{0}(t_0,x)-{\bf u}^{\infty}(t_0,x)\right)_{\nu_0} \right\|_{L^2(\P;\R)}
   \\&
   \leq
    \tfrac{7C^n2^{n-1}e^M}{\sqrt{M^{n-3}}}
    \left(
    \left[\sup_{(t,z)\in[t_0,T]\times \R^d}|(F(0))(t,z)|\right]
   + \left[\sup_{(t,z)\in [t_0,T]\times \R^d}\left\|{\bf u}^{\infty}(t,z)\right\|_\infty\right]  
    +\max\{\sqrt{T-t_0},\sqrt{3}\}\|K\|_1
    \right)
    \\&
     +(14 (4C)^{n-1}+1)\left[\sup_{(t,z)\in [t_0,T]\times \R^d
   }\|\eps(t,z)\|_\infty\right]
   .
  \end{split}     \end{equation}
\end{theorem}
\begin{proof}[Proof of Theorem~\ref{thm:rate}]
Lemma \ref{l:estimate.L2error} implies that
  \begin{equation}  \begin{split}\label{eq:non-recursive-bound}
   &  \left\| \left( {\bf U}_{n,M,Q}^{0}(t_0,x)-{\bf u}^{\infty}(t_0,x)\right)_{\nu_0}\right\|_{L^2(\P;\R)}
   \\&
   \leq
   \sum_{j=0}^{n-1}
   \sum_{\substack{l_1,\ldots,l_{j+1}\in\N,\\ l_1<\ldots<l_{j+1}=n}}
   \sum_{\substack{t_1,\ldots,t_j,t_{j+1}\in\R,\\t_0<t_1<\ldots<t_j<t_{j+1}\leq T}}
    \sum_{\nu_1,\ldots,\nu_{j+1} \in \{1,\ldots,d+1\}}
   \tfrac{2^j}{\sqrt{M^{n-j-l_1}}}
   \left[\prod_{i=1}^j L_{\nu_i}q^{Q,[t_{i-1},T]}(t_i)\right]
  \\&\quad\cdot
   \Bigg\{
   \1_{\{1\}}(\nu_{j+1})
   \Bigg(
    \1_{\{T\}}(t_{j+1})\Bigg(\left\|\left(\eps(t_j,x+W^0_{t_j}-W^0_{t_0})\right)_{\nu_j}\prod_{i=1}^j
\Big(   
   1,
     \tfrac{ 
     W^{0}_{t_i}- W^{0}_{t_{i-1}}
     }{ {t_{i}-t_{i-1}} } 
    \Big)_{\nu_{i-1}}
    \right\|_{L^2(\P;\R)}
    \\&\qquad \quad
    +\tfrac{1}{\sqrt{M^{l_1}}}
   \left\|\left(g(x+W^0_{T}-W^0_{t_0})-g(x+W^0_{t_j}-W^0_{t_0})\right)
   \prod_{i=1}^{j+1}
     \Big(   
   1,
     \tfrac{ 
     W^{0}_{t_i}- W^{0}_{t_{i-1}}
     }{ {t_{i}-t_{i-1}} } 
     \Big)_{\nu_{i-1}}
   \right\|_{L^2(\P;\R)}
   \Bigg)
   \\&\quad\qquad
   +
   \tfrac{q^{Q,[t_j,T]}(t_{j+1})}{\sqrt{M^{l_1}}}
   \left\|\left(F(0)\right)\!(t_{j+1},x+W_{t_{j+1}}^0-W_{t_0}^0)
   \prod_{i=1}^{j+1}
     \Big(   
   1,
     \tfrac{ 
     W^{0}_{t_i}- W^{0}_{t_{i-1}}
     }{ {t_{i}-t_{i-1}} } 
     \Big)
     _{\nu_{i-1}}
   \right\|_{L^2(\P;\R)}\Bigg)
   \\&\quad\qquad
   +
   \tfrac{L_{\nu_{j+1}}q^{Q,[t_j,T]}(t_{j+1})}{\sqrt{M^{l_1-1}}}
   \left\|\left({\bf u}^{\infty}(t_{j+1},x+W_{t_{j+1}}^0-W_{t_0}^0)\right)_{\nu_{j+1}}
   \prod_{i=1}^{j+1}
     \Big(   
   1,
     \tfrac{ 
     W^{0}_{t_i}- W^{0}_{t_{i-1}}
     }{ {t_{i}-t_{i-1}} } 
     \Big)_{\nu_{i-1}}
   \right\|_{L^2(\P;\R)}
   \Bigg\}
   .
  \end{split}     \end{equation}
  This, \eqref{eq:gLipschitz} and independence of Brownian increments prove that
  \begin{equation} \label{eq:aux_glob1}
   \begin{split}
   &  \left\| \left( {\bf U}_{n,M,Q}^{0}(t_0,x)-{\bf u}^{\infty}(t_0,x)\right)_{\nu_0}\right\|_{L^2(\P;\R)}
   \\&
   \leq
   \sum_{j=0}^{n-1}
   \sum_{\substack{l_1,\ldots,l_{j+1}\in\N,\\ l_1<\ldots<l_{j+1}=n}}
   \sum_{\substack{t_1,\ldots,t_j,t_{j+1}\in\R,\\t_0<t_1<\ldots<t_j<t_{j+1}\leq T}}
    \sum_{\nu_1,\ldots,\nu_{j+1} \in \{1,\ldots,d+1\}}
   \tfrac{2^j}{\sqrt{M^{n-j-l_1}}}
   \left[\prod_{i=1}^j L_{\nu_i}q^{Q,[t_{i-1},T]}(t_i)\right]
  \\&\quad\cdot
   \Bigg\{
   \1_{\{1\}}(\nu_{j+1})
   \Bigg(
    \1_{\{T\}}(t_{j+1})\Bigg(\left[\sup_{(t,z)\in [t_0,T]\times \R^d}\left(\eps(t,z)\right)_{\nu_j}\right]\prod_{i=1}^j\left\|
\Big(   
   1,
     \tfrac{ 
     W^{0}_{t_i}- W^{0}_{t_{i-1}}
     }{ {t_{i}-t_{i-1}} } 
    \Big)_{\nu_{i-1}}
    \right\|_{L^2(\P;\R)}
    \\&\qquad \quad
    +\sum_{\alpha=1}^{d}\tfrac{K_\alpha}{\sqrt{M^{l_1}}}
    \left\|
    (W^0_T-W^0_{t_j})_{\alpha}
     \Big(   
   1,
     \tfrac{ 
     W^{0}_{T}- W^{0}_{t_{j}}
     }{ {T-t_{j}} } 
     \Big)_{\nu_{j}}
   \right\|_{L^2(\P;\R)}
   \prod_{i=1}^{j}\left\|
     \Big(   
   1,
     \tfrac{ 
     W^{0}_{t_i}- W^{0}_{t_{i-1}}
     }{ {t_{i}-t_{i-1}} } 
     \Big)_{\nu_{i-1}}
   \right\|_{L^2(\P;\R)}
   \Bigg)
   \\&\quad\qquad
   +
   \tfrac{q^{Q,[t_j,T]}(t_{j+1})\left[\sup_{(t,z)\in[t_0,T]\times \R^d}|(F(0))(t,z)|\right]}{\sqrt{M^{l_1}}}
   \prod_{i=1}^{j+1}
   \left\|
     \Big(   
   1,
     \tfrac{ 
     W^{0}_{t_i}- W^{0}_{t_{i-1}}
     }{ {t_{i}-t_{i-1}} } 
     \Big)
     _{\nu_{i-1}}
   \right\|_{L^2(\P;\R)}\Bigg)
   \\&\quad\qquad
   +
   \tfrac{L_{\nu_{j+1}}q^{Q,[t_j,T]}(t_{j+1})\left[\sup_{(t,z)\in [t_0,T]\times \R^d}\left|\left({\bf u}^{\infty}(t,z)\right)_{\nu_{j+1}}\right|\right]}{\sqrt{M^{l_1-1}}}
   \prod_{i=1}^{j+1}\left\|
     \Big(   
   1,
     \tfrac{ 
     W^{0}_{t_i}- W^{0}_{t_{i-1}}
     }{ {t_{i}-t_{i-1}} } 
     \Big)_{\nu_{i-1}}
   \right\|_{L^2(\P;\R)}
   \Bigg\}
   .
  \end{split}     \end{equation}
  It holds for all $\nu\in \{1,\ldots,d+1\}$, $t\in [0,T)$ that
  \begin{equation}
  \begin{split}
  &
  \sum_{\alpha=1}^{d}K_\alpha
    \left\|
    (W^0_T-W^0_{t})_{\alpha}
     \Big(   
   1,
     \tfrac{ 
     W^{0}_{T}- W^{0}_{t}
     }{ {T-t} } 
     \Big)_{\nu}
   \right\|_{L^2(\P;\R)}\\
   &=
   \sum_{\alpha=1}^{d}K_\alpha\left(\sqrt{T-t}\1_{\{1\}}(\nu)+\tfrac{\1_{[2,\infty)}(\nu)}{T-t}\|(W^0_T-W^0_t)_\alpha(W^0_T-W^0_t)_{\nu-1}\|_{L^2(\P;\R)}\right)\\
   &=\sqrt{T-t}\|K\|_1\1_{\{1\}}(\nu)+\tfrac{\1_{[2,\infty)}(\nu)}{T-t}\left(K_{\nu-1}\|(W^0_T-W^0_t)^2_{\nu-1}\|_{L^2(\P;\R)}+ \sum_{\alpha\in \{1,\ldots,d\}\setminus \{\nu-1\}}K_\alpha \|(W^0_T-W^0_t)_\alpha\|_{L^2(\P;\R)}^2\right)\\
   &=\sqrt{T-t}\|K\|_1\1_{\{1\}}(\nu)+\1_{[2,\infty)}(\nu)\left(\sqrt{3}K_{\nu-1}+ \sum_{\alpha\in \{1,\ldots,d\}\setminus \{\nu-1\}}K_\alpha\right)\\
   &\leq \max\{\sqrt{T-t},\sqrt{3}\}\|K\|_1.
   \end{split}
  \end{equation}
  This and \eqref{eq:aux_glob1} show that
  \begin{equation} \label{eq:aux_glob2} \begin{split}
   &  \left\| \left( {\bf U}_{n,M,Q}^{0}(t_0,x)-{\bf u}^{\infty}(t_0,x)\right)_{\nu_0}\right\|_{L^2(\P;\R)}
   \\&
   \leq
   \left[\sup_{(t,z)\in [t_0,T]\times \R^d}\left(\eps(t,z)\right)_{\nu_0}\right]
    +\tfrac{\max\{\sqrt{T-t_0},\sqrt{3}\}\|K\|_1}{\sqrt{M^{n}}}\\
    &
    +
    \tfrac{\left[\sup_{(t,z)\in[t_0,T]\times \R^d}|(F(0))(t,z)|\right]}{\sqrt{M^{n}}}
    \sum_{t_1\in(t_0,T]}
    q^{Q,[t_0,T]}(t_{1})
   \left\|
     \Big(   
   1,
     \tfrac{ 
     W^{0}_{t_1}- W^{0}_{t_{0}}
     }{ {t_{1}-t_{0}} } 
     \Big)
     _{\nu_{0}}
   \right\|_{L^2(\P;\R)}
    \\
   &
   +
   \tfrac{\left[\sum_{\nu_1=1}^{d+1}L_{\nu_1}\sup_{(t,z)\in [t_0,T]\times \R^d}\left|\left({\bf u}^{\infty}(t,z)\right)_{\nu_{1}}\right|\right]}{\sqrt{M^{n-1}}}
   \sum_{t_1\in(t_0,T]}
   q^{Q,[t_0,T]}(t_{1})
   \left\|
     \Big(   
   1,
     \tfrac{ 
     W^{0}_{t_1}- W^{0}_{t_{0}}
     }{ {t_{1}-t_{0}} } 
     \Big)_{\nu_{0}}
   \right\|_{L^2(\P;\R)}
   \\&
   +
   \sum_{j=1}^{n-1}
   \sum_{\substack{l_1,\ldots,l_{j}\in\N,\\ l_1<\ldots<l_{j}<n}}
   \sum_{\substack{t_1,\ldots,t_j,t_{j+1}\in\R,\\t_0<t_1<\ldots<t_j<t_{j+1}\leq T}}
    \sum_{\nu_1,\ldots,\nu_{j+1} \in \{1,\ldots,d+1\}}
   \tfrac{2^j}{\sqrt{M^{n-j-l_1}}}
   \left[\prod_{i=1}^j L_{\nu_i}q^{Q,[t_{i-1},T]}(t_i)\right]
  \\&\quad\cdot
   \Bigg\{
   \1_{\{1\}}(\nu_{j+1})
   \Bigg(
    \1_{\{T\}}(t_{j+1})\Bigg(\left[\sup_{(t,z)\in [t_0,T]\times \R^d}\left(\eps(t,z)\right)_{\nu_j}\right]\prod_{i=1}^j\left\|
\Big(   
   1,
     \tfrac{ 
     W^{0}_{t_i}- W^{0}_{t_{i-1}}
     }{ {t_{i}-t_{i-1}} } 
    \Big)_{\nu_{i-1}}
    \right\|_{L^2(\P;\R)}
    \\&\qquad \quad
    +\tfrac{\max\{\sqrt{T-t_0},\sqrt{3}\}\|K\|_1}{\sqrt{M^{l_1}}}
   \prod_{i=1}^{j}\left\|
     \Big(   
   1,
     \tfrac{ 
     W^{0}_{t_i}- W^{0}_{t_{i-1}}
     }{ {t_{i}-t_{i-1}} } 
     \Big)_{\nu_{i-1}}
   \right\|_{L^2(\P;\R)}
   \Bigg)
   \\&\quad\qquad
   +
   \tfrac{q^{Q,[t_j,T]}(t_{j+1})\left[\sup_{(t,z)\in[t_0,T]\times \R^d}|(F(0))(t,z)|\right]}{\sqrt{M^{l_1}}}
   \prod_{i=1}^{j+1}
   \left\|
     \Big(   
   1,
     \tfrac{ 
     W^{0}_{t_i}- W^{0}_{t_{i-1}}
     }{ {t_{i}-t_{i-1}} } 
     \Big)
     _{\nu_{i-1}}
   \right\|_{L^2(\P;\R)}\Bigg)
   \\&\quad\qquad
   +
   \tfrac{L_{\nu_{j+1}}q^{Q,[t_j,T]}(t_{j+1})\left[\sup_{(t,z)\in [t_0,T]\times \R^d}\left|\left({\bf u}^{\infty}(t,z)\right)_{\nu_{j+1}}\right|\right]}{\sqrt{M^{l_1-1}}}
   \prod_{i=1}^{j+1}\left\|
     \Big(   
   1,
     \tfrac{ 
     W^{0}_{t_i}- W^{0}_{t_{i-1}}
     }{ {t_{i}-t_{i-1}} } 
     \Big)_{\nu_{i-1}}
   \right\|_{L^2(\P;\R)}
   \Bigg\}
   .
  \end{split}     \end{equation}
  For all $j\in \N$, $\nu_0,\ldots,\nu_{j-1}\in \{1,\ldots,d+1\}$, and $t_1,\ldots,t_j \in \R$ satisfying $t_0<t_1<\ldots<t_j<T$ it holds that
  \begin{equation}
  \prod_{i=1}^{j}\left\|
     \Big(   
   1,
     \tfrac{ 
     W^{0}_{t_i}- W^{0}_{t_{i-1}}
     }{ {t_{i}-t_{i-1}} } 
     \Big)_{\nu_{i-1}}
   \right\|_{L^2(\P;\R)}
   =
   \prod_{i=1}^{j}
   \left[
    \1_{\{1\}}(\nu_{i-1})+\frac{\1_{[2,\infty)}(\nu_{i-1})}{\sqrt{t_i-t_{i-1}}}
    \right]
    \le 
    (\sqrt{T-t_0}+1)^j
    \prod_{i=1}^{j}
    \frac{1}{\sqrt{t_i-t_{i-1}}}.
  \end{equation}
This and \eqref{eq:aux_glob2} ensure that
  \begin{equation} \label{eq:aux_glob3} \begin{split}
   &  \left\| \left( {\bf U}_{n,M,Q}^{0}(t_0,x)-{\bf u}^{\infty}(t_0,x)\right)_{\nu_0}\right\|_{L^2(\P;\R)}
   \\&
   \leq
   \left[\sup_{(t,z)\in [t_0,T]\times \R^d}\left(\eps(t,z)\right)_{\nu_0}\right]
    +\tfrac{\max\{\sqrt{T-t_0},\sqrt{3}\}\|K\|_1}{\sqrt{M^{n}}}
    +
    \tfrac{(\sqrt{T-t_0}+1)\left[\sup_{(t,z)\in[t_0,T]\times \R^d}|(F(0))(t,z)|\right]}{\sqrt{M^{n}}}
    \sum_{t_1\in(t_0,T]}
   \tfrac{q^{Q,[t_0,T]}(t_{1})}{\sqrt{t_1-t_{0}}}
    \\
   &
   +
   \tfrac{(\sqrt{T-t_0}+1)\left[\sup_{(t,z)\in [t_0,T]\times \R^d}\left\|{\bf u}^{\infty}(t,z)\right\|_\infty\right]\sum_{\nu_1=1}^{d+1}L_{\nu_1}}{\sqrt{M^{n-1}}}
    \sum_{t_1\in(t_0,T]}
   \tfrac{q^{Q,[t_0,T]}(t_{1})}{\sqrt{t_1-t_{0}}}
   \\&
   +
   \sum_{j=1}^{n-1}
   \sum_{\substack{l_1,\ldots,l_{j}\in\N,\\ l_1<\ldots<l_{j}<n}}
   \vast\{
    \tfrac{2^j(\sqrt{T-t_0}+1)^j\left[\sup_{(t,z)\in [t_0,T]\times \R^d
   }\|\eps(t,z)\|_\infty\right]}{\sqrt{M^{n-j-l_1}}}
   \left[
    \sum_{\substack{t_1,\ldots,t_j\in\R,\\t_0<t_1<\ldots<t_j< T}}    
    \prod_{i=1}^{j}
    \tfrac{q^{Q,[t_{i-1},T]}(t_i)}{\sqrt{t_i-t_{i-1}}}
    \right]
    \left[
    \sum_{\nu_1,\ldots,\nu_{j} \in \{1,\ldots,d+1\}}
    \prod_{i=1}^{j}L_{\nu_i}
    \right]
    \\&
    +\tfrac{2^j (\sqrt{T-t_0}+1)^j\max\{\sqrt{T-t_0},\sqrt{3}\}\|K\|_1}{\sqrt{M^{n-j}}}
    \left[
    \sum_{\substack{t_1,\ldots,t_j\in\R,\\t_0<t_1<\ldots<t_j< T}}    
    \prod_{i=1}^{j}
    \tfrac{q^{Q,[t_{i-1},T]}(t_i)}{\sqrt{t_i-t_{i-1}}}
    \right]
    \left[
    \sum_{\nu_1,\ldots,\nu_{j} \in \{1,\ldots,d+1\}}
    \prod_{i=1}^{j}L_{\nu_i}
    \right]
   \\&
   +
   \tfrac{2^j(\sqrt{T-t_0}+1)^{j+1}\left[\sup_{(t,z)\in[t_0,T]\times \R^d}|(F(0))(t,z)|\right]}{\sqrt{M^{n-j}}}
   \left[
   \sum_{\substack{t_1,\ldots,t_j,t_{j+1}\in\R,\\t_0<t_1<\ldots<t_j<t_{j+1}\leq T}}
   \prod_{i=1}^{j+1}
   \tfrac{q^{Q,[t_{i-1},T]}(t_i)}{\sqrt{t_i-t_{i-1}}}
   \right]
    \left[
    \sum_{\nu_1,\ldots,\nu_{j} \in \{1,\ldots,d+1\}}
    \prod_{i=1}^{j}L_{\nu_i}
    \right]
   \\&
   +
   \tfrac{2^j(\sqrt{T-t_0}+1)^{j+1}\left[\sup_{(t,z)\in [t_0,T]\times \R^d}\|{\bf u}^{\infty}(t,z)\|_\infty\right]}{\sqrt{M^{n-j-1}}}
   \left[
   \sum_{\substack{t_1,\ldots,t_j,t_{j+1}\in\R,\\t_0<t_1<\ldots<t_j<t_{j+1}\leq T}}
   \prod_{i=1}^{j+1}
   \tfrac{q^{Q,[t_{i-1},T]}(t_i)}{\sqrt{t_i-t_{i-1}}}
   \right]
    \left[
    \sum_{\nu_1,\ldots,\nu_{j+1} \in \{1,\ldots,d+1\}}
    \prod_{i=1}^{j+1}L_{\nu_i}
    \right]
   \vast\}
   .
  \end{split}     \end{equation}  
  Observe that for all $j\in \N$ it holds that
  \begin{equation}
  \left[
    \sum_{\nu_1,\ldots,\nu_{j} \in \{1,\ldots,d+1\}}
    \prod_{i=1}^{j}L_{\nu_i}
    \right]=\|L\|_1^j.
  \end{equation}
This, Lemma \ref{l:ub_iterated_GL}, \eqref{eq:aux_glob3}, and the fact that $\Gamma(\frac{1}{2})=\sqrt{\pi}$ imply that
  \begin{equation} \label{eq:aux_glob4} \begin{split}
   &  \left\| \left( {\bf U}_{n,M,Q}^{0}(t_0,x)-{\bf u}^{\infty}(t_0,x)\right)_{\nu_0}\right\|_{L^2(\P;\R)}
   \\&
   \leq
  \left[ \sup_{(t,z)\in [t_0,T]\times \R^d}\left(\eps(t,z)\right)_{\nu_0}\right]
    +\tfrac{\max\{\sqrt{T-t_0},\sqrt{3}\}\|K\|_1}{\sqrt{M^{n}}}
   \\
   &
   +
   \tfrac{2\sqrt{T-t_0}(\sqrt{T-t_0}+1)}{\sqrt{M^{n-1}}}
   \left(
   \tfrac{\left[\sup_{(t,z)\in[t_0,T]\times \R^d}|(F(0))(t,z)|\right]}{\sqrt{M}}
   +
   \|L\|_1\left[\sup_{(t,z)\in [t_0,T]\times \R^d}\left\|{\bf u}^{\infty}(t,z)\right\|_\infty\right]
   \right)
   \\&
   +2
   \sum_{j=1}^{n-1}
   \sum_{\substack{l_1,\ldots,l_{j}\in\N,\\ l_1<\ldots<l_{j}<n}}
   \Bigg\{
    \tfrac{2^j\|L\|_1^j(\sqrt{T-t_0}+1)^j\left(\sqrt{(T-t_0)\pi}\right)^j\left[\sup_{(t,z)\in [t_0,T]\times \R^d
   }\|\eps(t,z)\|_\infty\right]}{\sqrt{M^{n-j-l_1}}\Gamma(\frac{j}{2})}
    \\&
    +\tfrac{2^j\|L\|_1^j(\sqrt{T-t_0}+1)^j\left(\sqrt{(T-t_0)\pi}\right)^j\max\{\sqrt{T-t_0},\sqrt{3}\}\|K\|_1}{\sqrt{M^{n-j}}\Gamma(\frac{j}{2})}
   +
   \tfrac{2^j\|L\|_1^j(\sqrt{T-t_0}+1)^{j+1}\left(\sqrt{(T-t_0)\pi}\right)^{j+1}\left[\sup_{(t,z)\in[t_0,T]\times \R^d}|(F(0))(t,z)|\right]}{\sqrt{M^{n-j}}\Gamma(\frac{j+1}{2})}
   \\&
   +
   \tfrac{2^j\|L\|_1^{j+1}(\sqrt{T-t_0}+1)^{j+1}\left(\sqrt{(T-t_0)\pi}\right)^{j+1}\left[\sup_{(t,z)\in [t_0,T]\times \R^d}\|{\bf u}^{\infty}(t,z)\|_\infty\right]}{\sqrt{M^{n-j-1}}\Gamma(\frac{j+1}{2})}
   \Bigg\}
   .
  \end{split}     \end{equation}  
This, Lemma \ref{l:iterated_sum} and the definition \eqref{eq:defC1} of $C$ show that
\begin{equation} \label{eq:aux_glob5} \begin{split}
   &  \left\| \left( {\bf U}_{n,M,Q}^{0}(t_0,x)-{\bf u}^{\infty}(t_0,x)\right)_{\nu_0}\right\|_{L^2(\P;\R)}
   \\&
   \leq
   \left[\sup_{(t,z)\in [t_0,T]\times \R^d}\left(\eps(t,z)\right)_{\nu_0}\right]
    +\tfrac{\max\{\sqrt{T-t_0},\sqrt{3}\}\|K\|_1}{\sqrt{M^{n}}}
   +
   \tfrac{C\left(
  \left[\sup_{(t,z)\in[t_0,T]\times \R^d}|(F(0))(t,z)|\right]
   + \left[\sup_{(t,z)\in [t_0,T]\times \R^d}\left\|{\bf u}^{\infty}(t,z)\right\|_\infty\right]
   \right)}{\sqrt{\pi}\sqrt{M^{n-1}}}   
   \\&
   +\tfrac{2 \left[\sup_{(t,z)\in [t_0,T]\times \R^d
   }\|\eps(t,z)\|_\infty\right]}{\sqrt{M^n}}
   \sum_{j=1}^{n-1}
   \tfrac{(C\sqrt{M})^j}{\Gamma(\frac{j}{2})}
   \sum_{l_1=1}^{n-j}
   \sqrt{M}^{l_1}
   \binom{n-l_1-1}{j-1}
    \\&
    +\tfrac{2 \max\{\sqrt{T-t_0},\sqrt{3}\}\|K\|_1}{\sqrt{M^n}}
    \sum_{j=1}^{n-1}
    \tfrac{(C\sqrt{M})^j}{\Gamma(\frac{j}{2})}\binom{n-1}{j}
+\tfrac{C\left[\sup_{(t,z)\in[t_0,T]\times \R^d}|(F(0))(t,z)|\right]}{\sqrt{M^n}}
    \sum_{j=1}^{n-1}
    \tfrac{(C\sqrt{M})^j}{\Gamma(\frac{j+1}{2})}\binom{n-1}{j}
    \\&
+\tfrac{\left[\sup_{(t,z)\in [t_0,T]\times \R^d}\|{\bf u}^{\infty}(t,z)\|_\infty\right]}{\sqrt{M^n}}
    \sum_{j=1}^{n-1}
    \tfrac{(C\sqrt{M})^{j+1}}{\Gamma(\frac{j+1}{2})}\binom{n-1}{j}
    \\&
    \leq
   \left[\sup_{(t,z)\in [t_0,T]\times \R^d}\left(\eps(t,z)\right)_{\nu_0}\right]
    +\tfrac{\max\{\sqrt{T-t_0},\sqrt{3}\}\|K\|_1}{\sqrt{M^{n}}}
   \\
   &
   +
   \tfrac{C \left(
   \left[\sup_{(t,z)\in[t_0,T]\times \R^d}|(F(0))(t,z)|\right]
   + \left[\sup_{(t,z)\in [t_0,T]\times \R^d}\left\|{\bf u}^{\infty}(t,z)\right\|_\infty\right]
   \right)}{\sqrt{M^{n-1}}}
  \sum_{j=0}^{n-1}
    \tfrac{(C\sqrt{M})^j}{\Gamma(\frac{j+1}{2})}\binom{n-1}{j}
   \\&
   +\tfrac{2 \left[\sup_{(t,z)\in [t_0,T]\times \R^d
   }\|\eps(t,z)\|_\infty\right]}{\sqrt{M^n}}
   \sum_{j=1}^{n-1}
   \tfrac{(C\sqrt{M})^j}{\Gamma(\frac{j}{2})}
   \sum_{l=1}^{n-j}
   \sqrt{M}^l
   \binom{n-l-1}{j-1}
    \\&
    +\tfrac{2 \max\{\sqrt{T-t_0},\sqrt{3}\}\|K\|_1}{\sqrt{M^n}}
    \sum_{j=1}^{n-1}
    \tfrac{(C\sqrt{M})^j}{\Gamma(\frac{j}{2})}\binom{n-1}{j}
   .
  \end{split}     \end{equation}  
  It holds for all $r\in [0,\infty)$ that
  \begin{equation}
  \begin{split}\label{eq:ub_sum_gen}
     \sum_{j=0}^{n-1}\tfrac{r^j}{\Gamma(\frac{j+1}{2})}
     &\leq 
     \tfrac{1}{\sqrt{\pi}}+ \sum _{j=1}^{n-1}\tfrac{r^j}{\Gamma(\lfloor\frac{j+1}{2}\rfloor)}
     = 
     \tfrac{1}{\sqrt{\pi}}+ \sum_{l=1}^{\lfloor \frac{n}{2} \rfloor} \frac{r^{2l-1}}{\Gamma(l)}
     +\sum_{l=1}^{\lfloor \frac{n-1}{2} \rfloor} \frac{r^{2l}}{\Gamma(l)}\\
     &= 
     \tfrac{1}{\sqrt{\pi}}+ \sum_{l=0}^{\lfloor \frac{n}{2} \rfloor-1} \frac{r^{2l+1}}{l!}
     +\sum_{l=0}^{\lfloor \frac{n-1}{2} \rfloor-1} \frac{r^{2l+2}}{l!} \leq 
     \tfrac{1}{\sqrt{\pi}}+ r(r+1)e^{r^2}.
     \end{split}
  \end{equation}
  Note that it holds for all $j\in \{0,\ldots,n-1\}$ that 
$
\binom{n-1}{j}\leq \sum_{k=0}^{n-1} \binom{n-1}{k}=2^{n-1}.
$ 
  This and \eqref{eq:ub_sum_gen} ensure that
  \begin{equation}
  \begin{split}\label{eq:ub_sum1}
  \sum_{j=0}^{n-1}
    \tfrac{(C\sqrt{M})^j}{\Gamma(\frac{j+1}{2})}\binom{n-1}{j} &\leq 
(2C)^{n-1}    
 \sum_{j=0}^{n-1}
    \tfrac{\sqrt{M}^j}{\Gamma(\frac{j+1}{2})}
    \leq
    (2C)^{n-1}\left(
    \tfrac{1}{\sqrt{\pi}}+ \sqrt{M}(\sqrt{M}+1)e^{M}
    \right)\\
   &\leq 3(2C)^{n-1}Me^M
  \end{split}
  \end{equation}
 and that
  \begin{equation}
  \begin{split}\label{eq:ub_sum2}
  \sum_{j=1}^{n-1}
    \tfrac{(C\sqrt{M})^j}{\Gamma(\frac{j}{2})}\binom{n-1}{j} &\leq 
    (2C)^{n-1}\sqrt{M}\sum_{j=0}^{n-1} \tfrac{\sqrt{M}^j}{\Gamma(\frac{j+1}{2})}
    \leq
    3(2C)^{n-1}\sqrt{M}^3e^M.
   \end{split}
  \end{equation}
  For all $j\in \{1,\ldots, n-1\}$ it holds that
  \begin{equation}
  \begin{split}
   \sum_{l=1}^{n-j}
   \sqrt{M}^l
   \binom{n-l-1}{j-1} 
   &=
   \sum_{l=j-1}^{n-2}\sqrt{M}^{n-l-1}\binom{l}{j-1}
    \leq \sqrt{M}^{n-1} \sum_{l=j-1}^{\infty}\left(\frac{1}{\sqrt{M}}\right)^{l}\binom{l}{j-1} \\
    &= \frac{\sqrt{M}^{n-1}\left(\frac{1}{\sqrt{M}}\right)^{j-1}}{\left(1-\frac{1}{\sqrt{M}}\right)^{j}}
    =\frac{\sqrt{M}^{n-j}}{\left(1-\frac{1}{\sqrt{M}}\right)^{j}}.
    \end{split}
  \end{equation}
This together with \eqref{eq:ub_sum_gen} ensures that
  \begin{equation}
  \begin{split}\label{eq:ub_sum3}
  \sum_{j=1}^{n-1}
   \tfrac{(C\sqrt{M})^j}{\Gamma(\frac{j}{2})}
   \sum_{l=1}^{n-j}
   \sqrt{M}^l
   \binom{n-l-1}{j-1} 
   &\leq \sqrt{M}^n
   \sum_{j=1}^{n-1}
   \tfrac{C^j}{\Gamma(\frac{j}{2})\left(1-\frac{1}{\sqrt{M}}\right)^{j}}
   \leq \sqrt{M}^n\tfrac{C^{n-1}}{(1-\frac{1}{\sqrt 2})^{n-1}}\sum_{j=1}^{n-1}\frac{1}
   {\Gamma(\frac{j}{2})}
   \\
   &\leq (4C)^{n-1}\sqrt{M}^n 
   \left(\tfrac{1}{\sqrt{\pi}}+2e\right)
\leq 7 (4C)^{n-1}\sqrt{M}^n 
.
\end{split}
\end{equation}
  Combining \eqref{eq:aux_glob5}, \eqref{eq:ub_sum1}, \eqref{eq:ub_sum2}, and \eqref{eq:ub_sum3}
  proves that
  \begin{equation} \label{eq:aux_glob6} \begin{split}
   &  \left\| \left( {\bf U}_{n,M,Q}^{0}(t_0,x)-{\bf u}^{\infty}(t_0,x)\right)_{\nu_0}\right\|_{L^2(\P;\R)}
   \\&
    \leq
   \left[\sup_{(t,z)\in [t_0,T]\times \R^d}\left(\eps(t,z)\right)_{\nu_0}\right]
    +\tfrac{\max\{\sqrt{T-t_0},\sqrt{3}\}\|K\|_1}{\sqrt{M^{n}}}\\
    &
   +
   \tfrac{3C^n2^{n-1}e^M}{\sqrt{M^{n-3}}} 
   \left(
    \left[\sup_{(t,z)\in[t_0,T]\times \R^d}|(F(0))(t,z)|\right]
   + \left[\sup_{(t,z)\in [t_0,T]\times \R^d}\left\|{\bf u}^{\infty}(t,z)\right\|_\infty\right]
   \right) 
   \\&
   +14 (4C)^{n-1}\left[\sup_{(t,z)\in [t_0,T]\times \R^d
   }\|\eps(t,z)\|_\infty\right]
    +\tfrac{6(2C)^{n-1}e^M \max\{\sqrt{T-t_0},\sqrt{3}\}\|K\|_1}{\sqrt{M^{n-3}}}
     \\&
    \leq
    \tfrac{7C^n2^{n-1}e^M}{\sqrt{M^{n-3}}}
    \left(
    \left[\sup_{(t,z)\in[t_0,T]\times \R^d}|(F(0))(t,z)|\right]
   + \left[\sup_{(t,z)\in [t_0,T]\times \R^d}\left\|{\bf u}^{\infty}(t,z)\right\|_\infty\right]  
    +\max\{\sqrt{T-t_0},\sqrt{3}\}\|K\|_1
    \right)
    \\&
     +(14 (4C)^{n-1}+1)\left[\sup_{(t,z)\in [t_0,T]\times \R^d
   }\|\eps(t,z)\|_\infty\right]
   .
  \end{split}     \end{equation}
  This completes the proof of Theorem~\ref{thm:rate}.
  \end{proof}
\begin{lemma}[Quadrature error]\label{l:quad_error_glob}
Assume the setting in Section~\ref{sec:setting.full.discretization},
let $p,Q\in \N$, $x\in \R^d$, $s\in [0,T)$, and assume that $u^\infty \in C^\infty([0,T]\times \R^d,\R)$ and for all $k\in \N_0$ that
\begin{equation}\label{eq:ldeltaintegrandk}
   \sup_{(t,y)\in [0,T]\times \R^d}\frac{\left|\left((\tfrac{\partial}{\partial r}+\tfrac{1}{2}\Delta_y)^k u^{\infty}\right)\!(t,y)\right|}{1+\|y\|_1^p}<\infty.
 \end{equation}
 Then there exists $\xi\in [s,T]^{d+1}$ such that for all $\nu \in \{1,\ldots, d+1\}$ it holds that
\begin{equation}\label{eq:quad_error_glob}
 \begin{split}
   &
    \E\!\left[
    \sum_{t\in(s,T)}q^{Q,[s,T]}(t)
    \left(\funcF({\bf u}^{\infty})\right)\!(t,x+W_{t-s}^0)
     \Big(
     1 ,
     \tfrac{ 
     W^{0}_{t-s}
          }{ {t - s} }
     \Big)_{\nu}
     -
      \smallint_s^{T}
      (\funcF({\bf u}^{\infty}))(t,x+W_{t-s}^0)
     \Big(
     1 ,
     \tfrac{ 
     W^{0}_{t-s}
     }{ t- s }
     \Big)_{\nu}
      \,dt
    \right]
      \\
      &=
   (1,\nabla_x)_\nu\E\left[
    \left((\tfrac{\partial}{\partial r}+\tfrac{1}{2}\Delta_y)^{2Q+1}u^{\infty}\right)\!(\xi_\nu,x+W_{\xi_\nu}^0-W_s^0)
 \right]
    \tfrac{[Q!]^4(T-s)^{2Q+1}}{(2Q+1)[(2Q)!]^3}.
 \end{split}
 \end{equation}   
\end{lemma}
\begin{proof}[Proof of Lemma \ref{l:quad_error_glob}]
Observe that \eqref{eq:ldeltaintegrandk} and the dominated convergence theorem ensure that for every 
$k\in\N_0$ it holds that
 the function
 \begin{equation}  \begin{split}\label{eq:integrandk}
   [s,T]\ni t\mapsto
 \E\!\left[\left((\tfrac{\partial}{\partial r}+\tfrac{1}{2}\Delta_y)^k u^{\infty}\right)\!(t,x+W_{t-s}^0)
 \right]\in\R
 \end{split}     \end{equation}
 is continuous.
  The assumption that $u^{\infty}\in C^{\infty}([0,T]\times\R^d,\R)$
  and
  It\^o's formula imply that 
  for all $t\in[s,T]$, $k\in\N$ it holds $\P$-a.s.\ that
  \begin{align}\label{eq:cor.after.ito}
      &\left((\tfrac{\partial}{\partial r}
            +\tfrac{1}{2}\Delta_y
            )^k
      u^{\infty}\right)\!(t,x+W_{t}^{0}-W_s^0)
      -
      \left((\tfrac{\partial}{\partial r}
            +\tfrac{1}{2}\Delta_y
            )^k
       u^{\infty}\right)\!(s,x)
    \\&
    =\int_s^t
    \left((\tfrac{\partial}{\partial r}+\tfrac{1}{2}\Delta_y)^{k+1} u^{\infty}\right)\!(v,x+W_{v}^0-W_s^0)\,dv
    +\int_s^t
    \left\langle
    \left(\nabla_y(\tfrac{\partial}{\partial r}+\tfrac{1}{2}\Delta_y)^{k}u^{\infty}\right)\!(v,x+W_{v}^0-W_s^0),
    \,dW_v^0
    \right\rangle.\nonumber
  \end{align}
  This and \eqref{eq:ldeltaintegrandk} show that
 for all $k\in\N$ it holds that
  $\E\big[\sup_{t\in[s,T]}\big|
    \int_s^t
    \left\langle
    \left(\nabla_y(\tfrac{\partial}{\partial r}+\tfrac{1}{2}\Delta_y)^{k}u^{\infty}\right)(v,x+W_{v}^0-W_s^0),
    \,dW_v^0
    \right\rangle
   \big|\big]<\infty$.
 This implies that
 for all $t\in[s,T]$, $k\in\N$ it holds that
  $\E\big[
    \int_s^t
    \left\langle
    \left(\nabla_y(\tfrac{\partial}{\partial r}+\tfrac{1}{2}\Delta_y)^{k}u^{\infty}\right)(v,x+W_{v}^0-W_s^0),
    \,dW_v^0
    \right\rangle
   \big]=0$.
   This, \eqref{eq:cor.after.ito}, and Fubini's theorem show 
   that
  for all $t\in[s,T]$, $k\in\N$ it holds that
  \begin{equation}  \begin{split}\label{eq:exp.cor.after.ito}
    &
     \E\!\left[
      \left((\tfrac{\partial}{\partial r}
            +\tfrac{1}{2}\Delta_y
            )^k
      u^{\infty}\right)\!(t,x+W_{t}^{0}-W_s^0)
    \right]
      -
      \left((\tfrac{\partial}{\partial r}
            +\tfrac{1}{2}\Delta_y
            )^k
      u^{\infty}\right)\!(s,x)
    \\&
    =\int_s^t 
     \E\!\left[
    \left((\tfrac{\partial}{\partial r}+\tfrac{1}{2}\Delta_y)^{k+1}u^{\infty}\right)\!(v,x+W_{v}^0-W_s^0)
    \right]dv.
  \end{split}     \end{equation}
  Equation~\eqref{eq:exp.cor.after.ito} (with $k=1$)
  together with
  \eqref{eq:integrandk} (with $k=2$)
   implies 
  that the function
   $[s,T]\ni t\mapsto
   \E\!\left[\left((\tfrac{\partial}{\partial r}+\tfrac{1}{2}\Delta_y)u^{\infty}\right)(t,x+W_{t}^0-W_s^0)
 \right]\in\R$
 is continuously differentiable.
 Induction,~\eqref{eq:integrandk}, and
 \eqref{eq:exp.cor.after.ito}
  prove that
  it holds that the function 
  $[s,T]\ni t\mapsto
   \E\!\left[\left((\tfrac{\partial}{\partial r}+\tfrac{1}{2}\Delta_y)u^{\infty}\right)(t,x+W_{t}^0-W_s^0)
 \right]\in\R$ is infinitely often differentiable.
  This, 
 induction, and~\eqref{eq:exp.cor.after.ito}
 demonstrate that for all $k\in\N$, $t\in[s,T]$
 it holds that
 \begin{equation}  \begin{split}\label{eq:estimate.der}
    &\tfrac{\partial^{k}}{\partial t^{k}}
      \E\!\left[
      \left((\tfrac{\partial}{\partial r}+\tfrac{1}{2}\Delta_y)u^{\infty}\right)\!(t,x+W_{t}^0-W_s^0)
      \right]
    =
     \E\!\left[
    \left((\tfrac{\partial}{\partial r}+\tfrac{1}{2}\Delta_y)^{k+1}u^{\infty}\right)\!(t,x+W_{t}^0-W_s^0)
    \right]
    .
 \end{split}     \end{equation}
 Equation~\eqref{eq:PDE}
 and
 the error representation for the Gau\ss-Legendre quadrature rule
 (see, e.g., \cite[Display (2.7.12)]{davis2007methods})
 imply
 that
 there exists a real number $\xi_1\in[s,T]$ such that
 \begin{align}\label{eq:quad_error_1}
      &
    \sum_{t\in[s,T]}q^{Q,[s,T]}(t)
    \E\!\left[
       \left(\funcF( {\bf u^{\infty}})\right)\!(t,x+W_{t}^0-W_s^0)
\right]
      -
      \int_s^{T}
      \E\!\left[
      (\funcF({\bf u^{\infty}}))(t,x+W_{t}^0-W_s^0)
 \right]
      \,dt
   \\&
   =
     \int_s^{T}
      \E\!\left[
    \left((\tfrac{\partial}{\partial r}+\tfrac{1}{2}\Delta_y)u^{\infty}\right)\!(t,x+W_{t}^0-W_s^0)
 \right]
      \,dt
      -
    \sum_{t\in[s,T]}q^{Q,[s,T]}(t)
    \E\!\left[
    \left((\tfrac{\partial}{\partial r}+\tfrac{1}{2}\Delta_y)u^{\infty}\right)\!(t,x+W_{t}^0-W_s^0)
\right]
    \nonumber
\\&
    =\left(\tfrac{\partial^{2Q}}{\partial t^{2Q}}
      \E\left[
    \left((\tfrac{\partial}{\partial r}+\tfrac{1}{2}\Delta_y)u^{\infty}\right)\!(t,x+W_{t}^0-W_s^0)
 \right]
    \right)\Big|_{t=\xi_1}
    \tfrac{[Q!]^4(T-s)^{2Q+1}}{(2Q+1)[(2Q)!]^3}
    \nonumber
    \\&
    =
      \E\left[
    \left((\tfrac{\partial}{\partial r}+\tfrac{1}{2}\Delta_y)^{2Q+1}u^{\infty}\right)\!(\xi_1,x+W_{\xi_1}^0-W_s^0)
 \right]
    \tfrac{[Q!]^4(T-s)^{2Q+1}}{(2Q+1)[(2Q)!]^3}.
    \nonumber
 \end{align}
 Equation~\eqref{eq:PDE}, the Bismut-Elworthy-Li formula (see, e.g., \cite[Proposition 3.2]{Fournie1999})
 and
 the error representation for the Gau\ss-Legendre quadrature rule
 (see, e.g., \cite[Display (2.7.12)]{davis2007methods})
 imply
 for all $i\in \{1,\ldots,d\}$
 that
 there exists a real number $\xi_{i+1}\in[s,T]$ such that
 \begin{equation}\label{eq:quad_error_i}
 \begin{split}
   &
    \sum_{t\in[s,T]}q^{Q,[s,T]}(t)
    \E\!\left[
       \left(\funcF( {\bf u^{\infty}})\right)\!(t,x+W_{t}^0-W_s^0)
       \left(\tfrac{W^0_s-W^0_t}{s-t}\right)_i
\right]
      -
      \int_s^{T}
      \E\!\left[
      (\funcF({\bf u^{\infty}}))(t,x+W_{t}^0-W_s^0)
       \left(\tfrac{W^0_s-W^0_t}{s-t}\right)_i
 \right]
      \,ds
   \\&
   =\sum_{t\in[s,T]}q^{Q,[s,T]}(t)
    \tfrac{\partial}{\partial x_i}
    \E\!\left[
       \left(\funcF( {\bf u^{\infty}})\right)\!(t,x+W_{t}^0-W_s^0)
\right]
      -
      \int_s^{T}
       \tfrac{\partial}{\partial x_i}
      \E\!\left[
      (\funcF({\bf u^{\infty}}))(t,x+W_{t}^0-W_s^0)
 \right]
      \,dt
   \\&
   =
     \int_s^{T}
     \tfrac{\partial}{\partial x_i}
      \E\!\left[
    \left((\tfrac{\partial}{\partial r}+\tfrac{1}{2}\Delta_y)u^{\infty}\right)\!(t,x+W_{t}^0-W_s^0)
 \right]
      \,dt
      \\
      &\qquad \qquad
      -
    \sum_{t\in[s,T]}q^{Q,[s,T]}(t)
    \tfrac{\partial}{\partial x_i}
    \E\!\left[
    \left((\tfrac{\partial}{\partial r}+\tfrac{1}{2}\Delta_y)u^{\infty}\right)\!(t,x+W_{t}^0-W_s^0)
\right]
\\&
    =\left(\tfrac{\partial^{2Q}}{\partial t^{2Q}}
    \tfrac{\partial}{\partial x_i}
      \E\left[
    \left((\tfrac{\partial}{\partial r}+\tfrac{1}{2}\Delta_y)u^{\infty}\right)\!(t,x+W_{t}^0-W_s^0)
 \right]
    \right)\Big|_{t=\xi_i}
    \tfrac{[Q!]^4(T-s)^{2Q+1}}{(2Q+1)[(2Q)!]^3}
    \\&
    =
      \tfrac{\partial}{\partial x_i}\E\left[
    \left((\tfrac{\partial}{\partial r}+\tfrac{1}{2}\Delta_y)^{2Q+1}u^{\infty}\right)\!(\xi_i,x+W_{\xi_i}^0-W_s^0)
 \right]
    \tfrac{[Q!]^4(T-s)^{2Q+1}}{(2Q+1)[(2Q)!]^3}.
 \end{split}
 \end{equation}
 This and \eqref{eq:quad_error_1} prove \eqref{eq:quad_error_glob}. This completes the proof of Lemma \ref{l:quad_error_glob}.
\end{proof}  
  
\begin{corollary}\label{c:glob_error_nmq}
Assume the setting in Section~\ref{sec:setting.full.discretization},
assume that $u^\infty \in C^\infty([0,T]\times \R^d,\R)$,
let $n,Q\in \N$, $M\in \N\cap [2,\infty)$, $\nu_0\in \{1,\ldots,d+1\}$, $(t_0,x)\in [0,T)\times \R^d$, $\alpha \in [0,1]$ and let $C\in [0,\infty)$ be the real number given by
\begin{equation}
C=2(\sqrt{T-t_0}+1)\sqrt{(T-t_0)\pi}\left(\|L\|_1+1\right)+1.
\end{equation}
 Then it holds that
   \begin{equation} \begin{split}\label{eq:glob_error_nmq}
   &  \left\| \left( {\bf U}_{n,M,Q}^{0}(t_0,x)-{\bf u}^{\infty}(t_0,x)\right)_{\nu_0}\right\|_{L^2(\P;\R)}
   \\&
    \leq
  \tfrac{7C^n2^{n-1}e^M}{\sqrt{M^{n-3}}}
    \left(
    \left[\sup_{(t,z)\in[t_0,T]\times \R^d}|(F(0))(t,z)|\right]
   + \left[\sup_{(t,z)\in [t_0,T]\times \R^d}\left\|{\bf u}^{\infty}(t,z)\right\|_\infty\right]  
    +\max\{\sqrt{T-t_0},\sqrt{3}\}\|K\|_1
    \right)
    \\&
     +\tfrac{(14 (4C)^{n-1}+1)T^{2Q+1}}
     {Q^{2\alpha Q}}
   \left[\sup_{k\in \N}\sup_{(t,z)\in [t_0,T]\times \R^d}\tfrac{
   \left \|(1,\nabla_y)
    \left((\tfrac{\partial}{\partial r}+\tfrac{1}{2}\Delta_y)^{k}u^{\infty}\right)\!(t,z)\right \|_\infty}
    {(k!)^{1-\alpha}}
   \right]    
    .  
\end{split}
\end{equation}
 \end{corollary}
  
  \begin{proof}[Proof of Corollary \ref{c:glob_error_nmq}]
  To prove \eqref{eq:glob_error_nmq} we assume w.l.o.g.\ that the right-hand side of \eqref{eq:glob_error_nmq} is finite.
Observe that the Stirling-type formula in Robbins \cite[Displays (1)--(2)]{robbins1955remark} proves for all
$k\in \N$ that
\begin{equation}
\sqrt{2\pi k}\left[\frac{k}{e}\right]^k\le k! \le \sqrt{2\pi k}\left[\frac{k}{e}\right]^ke^{\frac{1}{12}}.
\end{equation}
This together with the fact that $e^2\le 8$ and the fact that $\forall \, k\in \N\colon \pi e^{\frac{1}{3}}k\le 8^{k}$ shows for all $k\in \N$ that
\begin{equation}\label{eq:stirling}
\begin{split}
\tfrac{k^{2\alpha k}((2k+1)!)^{1-\alpha}[k!]^4}{(2k+1)[(2k)!]^3}
&\le \tfrac{k^{2\alpha k}[k!]^4}{[(2k)!]^{2+\alpha}}
\le \tfrac{k^{2\alpha k}\left[\sqrt{2\pi}k^{k+\frac{1}{2}}e^{-k+\frac{1}{12}}\right]^4}
{\left[\sqrt{2\pi}(2k)^{2k+\frac{1}{2}}e^{-2k}\right]^{2+\alpha}}
=(\sqrt{2\pi})^{2-\alpha}k^{1-\frac{\alpha}{2}}e^{\frac{1}{3}}
2^{-(2k+\frac{1}{2})2}\big(\tfrac{e^{2k}}{2^{2k+\frac{1}{2}}}
\big)^{\alpha}
\\
&\le 2\pi k e^{\frac{1}{3}}2^{-4k-1}e^{2k}2^{-2k}
=\pi e^{\frac{1}{3}}k\big(\tfrac{e^2}{64}\big)^k
\le \pi e^{\frac{1}{3}}k8^{-k}
\le 1.
\end{split}
\end{equation}
 Theorem \ref{thm:rate} and Lemma \ref{l:quad_error_glob} ensure that
 \begin{equation} \begin{split}
   &  \left\| \left( {\bf U}_{n,M,Q}^{0}(t_0,x)-{\bf u}^{\infty}(t_0,x)\right)_{\nu_0}\right\|_{L^2(\P;\R)}
   \\&
    \leq
     \tfrac{7C^n2^{n-1}e^M}{\sqrt{M^{n-3}}}
    \left(
    \left[\sup_{(t,z)\in[t_0,T]\times \R^d}|(F(0))(t,z)|\right]
   + \left[\sup_{(t,z)\in [t_0,T]\times \R^d}\left\|{\bf u}^{\infty}(t,z)\right\|_\infty\right]  
    +\max\{\sqrt{T-t_0},\sqrt{3}\}\|K\|_1
    \right)
    \\&
     +(14 (4C)^{n-1}+1)
      \left[\sup_{(t,z)\in [t_0,T]\times \R^d}
   \left \|(1,\nabla_y)
    \left((\tfrac{\partial}{\partial r}+\tfrac{1}{2}\Delta_y)^{2Q+1}u^{\infty}\right)\!(t,z)\right \|_\infty
    \tfrac{[Q!]^4(T-t)^{2Q+1}}{(2Q+1)[(2Q)!]^3}
   \right]
  .
  \end{split}     \end{equation}  
   It follows that
   \begin{equation} \begin{split}
   &  \left\| \left( {\bf U}_{n,M,Q}^{0}(t_0,x)-{\bf u}^{\infty}(t_0,x)\right)_{\nu_0}\right\|_{L^2(\P;\R)}
   \\&
    \leq
 \tfrac{7C^n2^{n-1}e^M}{\sqrt{M^{n-3}}}
   \left(
    \left[\sup_{(t,z)\in[t_0,T]\times \R^d}|(F(0))(t,z)|\right]
   + \left[\sup_{(t,z)\in [t_0,T]\times \R^d}\left\|{\bf u}^{\infty}(t,z)\right\|_\infty\right]  
    +\max\{\sqrt{T-t_0},\sqrt{3}\}\|K\|_1
    \right)
    \\&
     +\tfrac{(14 (4C)^{n-1}+1)T^{2Q+1}}
     {Q^{2\alpha Q}}\left[
   \sup_{l\in\N}
    \tfrac{ l^{2\alpha l}((2l+1)!)^{1-\alpha}[l!]^4}{(2l+1)[(2l)!]^3}
    \right]
   \left[\sup_{k\in \N}\sup_{(t,z)\in [t_0,T]\times \R^d}\tfrac{
   \left \|(1,\nabla_y)
    \left((\tfrac{\partial}{\partial r}+\tfrac{1}{2}\Delta_y)^{k}u^{\infty}\right)\!(t,z)\right \|_\infty}
    {(k!)^{1-\alpha}}
   \right]    
 \\&
 \leq
   \tfrac{7C^n2^{n-1}e^M}{\sqrt{M^{n-3}}}
    \left(
    \left[\sup_{(t,z)\in[t_0,T]\times \R^d}|(F(0))(t,z)|\right]
   + \left[\sup_{(t,z)\in [t_0,T]\times \R^d}\left\|{\bf u}^{\infty}(t,z)\right\|_\infty\right]  
    +\max\{\sqrt{T-t_0},\sqrt{3}\}\|K\|_1
    \right)
    \\&
     +\tfrac{(14 (4C)^{n-1}+1)T^{2Q+1}}
     {Q^{2\alpha Q}}
   \left[\sup_{k\in \N}\sup_{(t,z)\in [t_0,T]\times \R^d}\tfrac{
   \left \|(1,\nabla_y)
    \left((\tfrac{\partial}{\partial r}+\tfrac{1}{2}\Delta_y)^{k}u^{\infty}\right)\!(t,z)\right \|_\infty}
    {(k!)^{1-\alpha}}
   \right]    
    .
  \end{split}     \end{equation}  
  This proves \eqref{eq:glob_error_nmq}. The proof of Corollary \ref{c:glob_error_nmq} is thus completed.
\end{proof} 
  
The following corollary (Corollary~\ref{c:glob_error_nnn})
specializes Corollary~\ref{c:glob_error_nmq}
to the special case $n=M=Q$ and $\alpha=\tfrac{1}{4}$.
For the choice of $\alpha$ note that the terms
$\sqrt{M}^{-n}$ and $Q^{-2\alpha Q}$ in the case
$n=M=Q\in\N\cap[2,\infty)$
are equal if and only if $\alpha=\tfrac{1}{4}$.
 \begin{corollary}\label{c:glob_error_nnn}
Assume the setting in Section~\ref{sec:setting.full.discretization},
assume that $u^\infty \in C^\infty([0,T]\times \R^d,\R)$,
let $n\in \N\cap [2,\infty)$, $\nu_0\in \{1,\ldots,d+1\}$, $(t_0,x)\in [0,T)\times \R^d$, and let $C\in [0,\infty)$ be the real number given by
\begin{equation}
C=2(\sqrt{T-t_0}+1)\sqrt{(T-t_0)\pi}\left(\|L\|_1+1\right)+1.
\end{equation}
 Then it holds that
   \begin{align}\label{eq:glob_error_nnn}
  &  \left\| \left( {\bf U}_{n,n,n}^{0}(t_0,x)-{\bf u}^{\infty}(t_0,x)\right)_{\nu_0}\right\|_{L^2(\P;\R)}
   \\&
    \leq
  \tfrac{7C^n2^{n-1}e^n}{\sqrt{n^{n-3}}}
    \left(
    \left[\sup_{(t,z)\in[t_0,T]\times \R^d}|(F(0))(t,z)|\right]
   + \left[\sup_{(t,z)\in [t_0,T]\times \R^d}\left\|{\bf u}^{\infty}(t,z)\right\|_\infty\right]  
    +\max\{\sqrt{T-t_0},\sqrt{3}\}\|K\|_1
    \right)
    \nonumber
    \\&
     +\tfrac{(14 (4C)^{n-1}+1)T^{2n+1}}
     {\sqrt{n^{n}}}
   \left[\sup_{k\in \N}\sup_{(t,z)\in [t_0,T]\times \R^d}\tfrac{
   \left \|(1,\nabla_y)
    \left((\tfrac{\partial}{\partial r}+\tfrac{1}{2}\Delta_y)^{k}u^{\infty}\right)\!(t,z)\right \|_\infty}
    {(k!)^{\nicefrac 34}}
   \right]    
    .  
    \nonumber
   \end{align}
 \end{corollary} 

The following main result of this article
(Corollary~\ref{c:computational_effort_glob_error})
proves that if the constant~\eqref{eq:constant_rate4} is finite,
then the computational complexity
(here measured in terms of the number of scalar normal random variables
and in terms of function evaluations of $f$ and $g$)
is bounded by $O(d\eps^{-(4+\delta)})$ for any $\delta\in(0,\infty)$
where $d$ is the dimensionality of the problem and $\eps\in(0,\infty)$
is the prescribed accuracy.
\begin{corollary}[Computational complexity in terms of
global error]\label{c:computational_effort_glob_error}
 Assume the setting in Subsection~\ref{sec:setting.full.discretization},
 assume that $u^\infty \in C^\infty([0,T]\times \R^d,\R)$,
 let $\delta \in (0,\infty)$, 
 let $C\in [0,\infty]$ be the extended
 real number
given by
 \begin{equation}\label{eq:constant_rate4}
\begin{split}
C&=
    \left(
    \sup_{(t,z)\in[0,T]\times \R^d}|(F(0))(t,z)|
    +\sqrt{T+3}\|K\|_1
   +
   \sup_{k\in \N_0}\tfrac{\sup_{(t,z)\in [0,T]\times \R^d}
   \left \|(1,\nabla_y)
    \left((\tfrac{\partial}{\partial r}+\tfrac{1}{2}\Delta_y)^{k}u^{\infty}\right)\!(t,z)\right \|_\infty}
    {(k!)^{\nicefrac 34}}
    \right)^{(4+\delta)},
  \end{split} \end{equation}
assume that $C<\infty$,
let $(\operatorname{RN}_{n,M,Q})_{n,M,Q\in\Z}\subseteq\N_0$ be natural numbers
which satisfy
for all $n,M,Q \in \N$
that $\operatorname{RN}_{0,M,Q}=0$
and
\begin{align}
  \operatorname{RN}_{ n,M,Q }
  &\leq d M^n+\sum_{l=0}^{n-1}\left[Q M^{n-l}( d + \operatorname{RN}_{ l, M,Q }+ \mathbbm{1}_{ \N }( l ) \cdot\operatorname{RN}_{ l-1, M,Q })\right]
\end{align}
(for every $N\in\N$ 
 we think of
 $\operatorname{RN}_{ N,N,N }$
 as the number of realizations of a scalar standard normal random variable required
 to compute one realization of the random variable
 $U^{0}_{N,N,N}(0,0)\colon\Omega\to\R$),
and let $( \operatorname{FE}_{n,M,Q})_{n,M,Q\in\Z}\subseteq\N_0$ be natural numbers
which satisfy
for all $n,M,Q \in \N$
that $\operatorname{FE}_{0,M,Q}=$
and
\begin{equation}  \begin{split}
  \operatorname{FE}_{ n,M,Q }
  &\leq M^n+\sum_{l=0}^{n-1}\left[Q M^{n-l}( 1 + \operatorname{FE}_{ l, M,Q }+ \mathbbm{1}_{ \N }( l )+ \mathbbm{1}_{ \N }( l ) \cdot\operatorname{FE}_{ l-1, M,Q })\right]
\end{split}     \end{equation}
(for every $N\in\N$ 
 we think of
 $\operatorname{FE}_{ N,N,N }$
 as the
 number of function evaluations of $f$ and $g$
 required
 to compute one realization of the random variable
 $U^{0}_{N,N,N}(0,0)\colon\Omega\to\R$).
 Then it holds for all $N\in \N$ that
 \begin{equation}  \begin{split}\label{eq:c.rate4}
 \operatorname{RN}_{N,N,N}+\operatorname{FE}_{N,N,N}
 \le &
 d \left[\sup_{(t,x)\in[0,T]\times\R^d}\max_{\nu\in\{1,\ldots,d+1\}}\left\|\left(
    {\bf U}_{N,N,N}^{0}(t,x)-{\bf u}^{\infty}(t,x)
    \right)_{\nu}
    \right\|_{L^2(\P;\R)}\right]^{-\left(4+\delta\right)}
    \\\cdot
    &16C
    \sum_{n\in\N}
    \big(24(T+1)\big)^{3(4+\delta)n}(\|L\|_{1}+1)^{(4+\delta)n}\sqrt{n}^{-\delta n}<\infty.
 \end{split}     \end{equation}
 \end{corollary} 
\begin{proof}[Proof of Corollary~\ref{c:computational_effort_glob_error}]
Lemma 3.15 and Lemma 3.16 in~\cite{EHutzenthalerJentzenKruse2016}
imply
that for all $N\in \N$ it holds that
$
\operatorname{RN}_{ N,N,N }
\leq 8 d N^{2N}
$
and 
$
\operatorname{FE}_{ N,N,N }
\leq  8 N^{2N}
$.
This  and Corollary~\ref{c:glob_error_nnn}
yield for all $N\in\N$ that
 \begin{equation} \begin{split}\label{eq:glob_error_comp_effort}
  &  (\textup{RN}_{N,N,N}+\textup{FE}_{N,N,N})\left[ \sup_{(t,x)\in[0,T]\times\R^d}\max_{\nu\in\{1,\ldots,d+1\}}
  \left\| \left( {\bf U}_{n,n,n}^{0}(t,x)-{\bf u}^{\infty}(t,x)\right)_{\nu_0}\right\|_{L^2(\P;\R)}
  \right]^{(4+\delta)}
   \\&
    \leq
    8(d+1) N^{2N}\cdot
    \left(
  \tfrac{7
    \left(
      2(\sqrt{T}+1)\sqrt{T\pi}\left(\|L\|_1+1\right)+1
    \right)^N2^{N-1}e^N}{\sqrt{N^{N-3}}}
     +\tfrac{(14 (
      8(\sqrt{T}+1)\sqrt{T\pi}\left(\|L\|_1+1\right)+4
     )^{N-1}+1)T^{2N+1}}{\sqrt{N^{N}}}
    \right)^{(4+\delta)}C
   \\&
    \leq
    8(d+1) N^{2N}\cdot
    \left(\left(24(T+1)\right)^{3N}(\|L\|_{1}+1)^N\sqrt{N}^{-N}\right)^{(4+\delta)}C
   \\&
    \leq
    16dC
    \sum_{n\in\N}
    \big(24(T+1)\big)^{3(4+\delta)n}(\|L\|_{1}+1)^{(4+\delta)n}\sqrt{n}^{-\delta n}
    .  
\end{split} \end{equation}
The right-hand side of~\eqref{eq:glob_error_comp_effort} 
is clearly finite.
This finishes the proof of Corollary~\ref{c:computational_effort_glob_error}.
\end{proof}

\subsubsection*{Acknowledgement}
This project has been partially supported 
by the Deutsche Forschungsgesellschaft (DFG) via RTG 2131 {\it High-dimensional Phenomena in Probability -- Fluctuations and Discontinuity}
and via research grant HU 1889/6-1.
 
\bibliographystyle{acm}
\bibliography{Bib/bibfile}

\end{document}